\documentclass[12pt]{amsart}
\usepackage{hyperref}
\hypersetup{colorlinks = true, linkcolor = {blue}}
\usepackage[all]{xy}
\usepackage{amssymb}
\usepackage{aliascnt}
\usepackage{verbatim}
\usepackage{tikz}
\usetikzlibrary{decorations.markings}
\usepackage[colorinlistoftodos]{todonotes}
\usepackage{amsmath}
\usepackage{amsthm}
\usepackage{latexsym,amsfonts,amssymb}
\usepackage{amsmath,amsthm,amscd}
\usepackage[dvips]{epsfig}
\usepackage{psfrag}

\usepackage{color}
\usepackage{etoolbox}
\usepackage{hyperref}
\usepackage{a4wide}
\usepackage[margin=1in]{geometry}
\usepackage{graphicx}
\usepackage{fancyhdr}
\usepackage{mathtools}
\usepackage{topcapt}
\usepackage{booktabs}
\usepackage{comment}
\usepackage[scaled]{helvet}
\usepackage[T1]{fontenc}
\usepackage{mathrsfs}
\input xy
\xyoption{all}

\newcommand{\ot}{\otimes}

\newcommand{\Z}{\mathbb{Z}}

\newcommand{\C}{\mathbb{C}}

\newcommand{\A}{\mathbb{A}}
\renewcommand{\P}{\mathbb{P}}
\newcommand{\inv}{^{-1}}
\newcommand{\Hom}{\text{\rm Hom}}
\newcommand{\Mat}{\text{\rm Mat}}

\newcommand{\GL}{\text{\rm GL}}

\newcommand{\PSL}{\text{\rm PSL}}
\newcommand{\SL}{\text{\rm SL}}

\renewcommand{\O}{\mathcal{O}}
\newcommand{\cF}{\mathcal{F}}

\newcommand{\cM}{\mathcal{M}}

\newcommand{\QCoh}{\mathsf{QCoh}}
\newcommand{\Spec}{\text{\rm Spec}}

\newcommand{\Proj}{\text{\rm Proj}}

\newcommand{\Sym}{\text{\rm Sym}}

\newcommand{\dmod}{\text{\sf -mod}}

\newcommand{\End}{\text{\rm End}}

\newcommand{\D}{\mathcal{D}}

\newcommand{\g}{\mathfrak{g}}
\newcommand{\Ug}{U\mathfrak{g}}
\newcommand{\OG}{\O(G)}

\newcommand{\frakt}{\mathfrak{t}}

\newcommand{\fraku}{\mathfrak{u}}
\newcommand{\frakl}{\mathfrak{l}}

\newcommand{\uProj}{\text{\sf Proj}}
\newcommand{\uproj}{\text{\sf proj}}
\newcommand{\GIT}{\text{ $\!$/$\! \!$/$\!$}}

\newcommand{\uqg}{U_q(\g)}
\newcommand{\OqG}{\O_q(G)}

\newcommand{\cV}{\mathcal{V}}
\newcommand{\Vinb}{\mathbb{V}}
\newcommand{\VinbG}{\mathbb{V}_G}
\newcommand{\VinbSL}{\mathbb{V}_{\SL_2}}
\newcommand{\barPSL}{\overline{\PSL_2}}
\newcommand{\Gad}{G^\text{\rm ad}}
\newcommand{\Gbarad}{\overline{G^\text{\rm ad}}}
\newcommand{\Gadbar}{\overline{G^\text{\rm ad}}}

\newcommand{\gr}{\text{\rm gr}}

\newcommand{\Rees}{\text{\rm Rees}}
\newcommand{\ReesG}{\text{\rm Rees}(G)}
\newcommand{\Orb}{\text{\rm Orb}} 
\newcommand{\Lad}{L_I^{\text{\rm ad}}}

\newcommand{\fraksl}{\mathfrak{sl}}
\newcommand{\Uqsl}{U_q(\fraksl_2)}
\newcommand{\Oqsl}{\O_q(\SL_2)}

\newcommand{\VinbqSL}{\O_q(\Vinb_{SL_2})}

\newcommand{\OqVinb}{\O_q(\Vinb_G)}
\newcommand{\uqulu}{ U_q(\mathfrak{u} \times \mathfrak{l} \times \mathfrak{u}^-) }

\newcommand{\Grmod}{\text{\sf Grmod}}
\newcommand{\grmod}{\text{\sf grmod}}
\newcommand{\Tors}{\text{\sf Tors}}
\newcommand{\tors}{\text{\sf tors}}

\theoremstyle{plain}
\newtheorem{theorem}{Theorem}[section]
\newtheorem{prop}[theorem]{Proposition}
\newtheorem{defprop}[theorem]{Definition-Proposition}
\newtheorem{lemma}[theorem]{Lemma}
\newtheorem{cor}[theorem]{Corollary}

\theoremstyle{definition}
\newtheorem{definition}[theorem]{Definition}
\newtheorem{example}[theorem]{Example}
\newtheorem{rmk}[theorem]{Remark}
\newtheorem{notation}[theorem]{Notation}

\numberwithin{equation}{section}

\begin{document} 
\title{The Wonderful Compactification for Quantum Groups}
\author{Iordan Ganev}
\address{The University of Texas at Austin, Department of Mathematics, Austin, Texas, USA 78712}
\maketitle

\parskip = 0pt 

\begin{abstract} In this paper, we introduce a quantum version of the wonderful compactification of a group as a certain noncommutative projective scheme. Our approach stems from the fact that the wonderful compactification encodes the asymptotics of matrix coefficients, and from its realization as a GIT quotient of the Vinberg semigroup. In order to define the wonderful compactification for a quantum group, we adopt a generalized formalism of $\mathsf{Proj}$ categories in the spirit of Artin and Zhang. Key to our construction is a quantum version of the Vinberg semigroup, which we define as a $q$-deformation of a certain Rees algebra, compatible with a standard Poisson structure. Furthermore, we discuss quantum analogues of the stratification of the wonderful compactification by orbits for a certain group action, and provide explicit computations in the case of $\mathrm{SL}_2$. \end{abstract}

\tableofcontents

\section{Introduction}\label{sec:intro}

New geometric perspectives on familiar representation-theoretic objects often illuminate deep unifying structures among diverse phenomena, and lead to the resolution of long-standing algebraic problems. This paper synthesizes two distinct sources of insight in representation theory: (1) the study of a group `at infinity,' that is, of compactifications of a group, and (2) the role of quantum groups as symmetries in noncommutative geometry. The interaction of these perspectives leads to results on the asymptotics of quantum groups, and to a deeper understanding of various central objects in representation theory.

To place our work in context, we recall that a given semisimple algebraic group has a distinguished compactification, known as the `wonderful' compactification. It is a projective variety, introduced by de Concini and Procesi, that  links the geometry of the group with the geometry of its partial flag varieties and Levi subgroups \cite{DCP}. In addition, it captures the equivariant degenerations of the group and encodes the asymptotics of matrix coefficients. Related to the wonderful compactification of a group is the associated Vinberg semigroup, introduced by Vinberg \cite{Vinberg}. The latter is an affine variety that (generically) forms a multi-cone over the wonderful compactification. As such, the Vinberg semigroup can be regarded as a linear version of the wonderful compactification where various structures simplify, and where the rational degenerations of the group become more apparent. 

The wonderful compactification has recently emerged as a powerful tool in geometric representation theory. For example, in the setting of $p$-adic groups, Bezrukavnikov and Kazhdan have illustrated how the wonderful compactification leads to a geometric understanding of the second adjointness theorem \cite{BezKazhdan} (see also \cite{SekVenkatesh}). Drinfeld and Gaitsgory use the Vinberg semigroup in order to establish adjunctions that relate categories of $\D$-modules on certain moduli stacks of bundles; their results contribute to the geometric Langlands program \cite{DrinfeldGaitsgory}. The Vinberg semigroup is also crucial in the proof of a dimension formula for a group version of affine Springer fibers, due to Bouthier \cite{Bouthier}. The wonderful compactification has been used in the theory of character sheaves by several authors \cite{BFO, GinzburgAMSS, He, Lusztig:2006, Springer}, while Lu, Yakimov, and others have studied the Poisson geometry of the wonderful compactification and related varieties \cite{LY1, LY2}.

Another, similarly fruitful, source of new perspectives in representation theory is the study of noncommutative geometry emanating from quantum groups. Since the inception of quantum groups by Drinfeld and Jimbo in the 1980s, much work has been devoted to the construction of $q$-deformations of classical varieties in order to understand categories of representations. These constructions often take the form of $q$-deformations of algebras (i.e.\ as global rather than formal quantizations) and pivot on the structure of the quantum group as a $q$-deformation of the universal enveloping algebra of a Lie algebra. Just like quantum groups themselves, the $q$-deformations that appear in representation theory have remarkable connections to various other areas of mathematics. For example, there are strong parallels between the behavior of an object's $q$-deformation when $q$ is a root of unity and the geometry of the classical object in positive characteristic. 

The spirit of quantum geometric representation theory is embodied in work of Backelin and Kremnitzer on quantum flag varieties and differential operators, as well as related results by Lunts and Rosenberg and by Tanisaki \cite{BaKr, LR, Tanisaki}. Their work demonstrates that the category of quasicoherent sheaves on the flag variety of a reductive group admits a $q$-deformation that can be regarded as the category of sheaves on the quantum flag variety, and is a noncommutative projective scheme in the sense of Artin and Zhang \cite{ArtinZhang}. Moreover, quantum versions of differential operators on flag varieties encode representations of quantum groups, giving rise to a quantum version of the Beilinson-Bernstein localization theorem. 

In this paper, we take inspiration from previous work on quantum flag varieties in order to establish the  wonderful compactification for quantum groups. Our approach is ultimately rooted in Peter-Weyl theorem and the asymptotics of matrix coefficients, and employs the formalism of noncommutative projective schemes. Moreover, we consider quantum versions of the wonderful compactification's stratification by $G \times G$ orbits and its associated Vinberg semigroup. Our perspectives relate directly to the aforementioned quantum flag varieties. 

In the next section, we give an overview of the key properties of the wonderful compactification and the Vinberg semigroup. In Section \ref{subsec:whatis}, we explain the context behind our construction of the wonderful compactification for quantum groups. We list the main results of this paper in Section \ref{subsec:mainresults}, before providing further motivation for our work in the form of on-going and future projects in Section \ref{subsec:motivation}.

\subsection{What is the wonderful compactification?}\label{subsec:whatis0} Let $G$ be a connected semisimple group over $\C$, and let $\Gad = G/Z(G)$ denote the adjoint group of $G$. The wonderful compactification is a certain projective variety $\Gadbar$ that contains $\Gad$ as an open subvariety. We recall the precise definition of $\Gadbar$ in Section \ref{sec:classical}. For the purposes of this overview, here we simply highlight the key properties of $\Gadbar$:

\begin{enumerate}
 \item The variety $\Gadbar$ is stratified by the orbits of a $G \times G$ action. These orbits are indexed by subsets $I$ of the set $\Delta$ of positive simple roots. 
 
  \item The orbit corresponding to $I = \Delta$ is a copy of the adjoint group $\Gad$ with $G \times G$ action by left and right multiplication. This is the unique open orbit.
  
  \item The orbit corresponding to $I = \emptyset$ is the square of the flag variety, $G/B \times B^-\backslash G$, with $G \times G$ action induced by outermost left and right multiplication. This is the unique closed orbit. 
  
  \item The other orbits are related to partial flag varieties, Levi subgroups, rational degenerations of $G$, and wonderful compactifications of groups of smaller rank. The rich structure of the orbits distinguishes the wonderful compactification from other compactifications\footnote{The adjective `wonderful' is a technical term. A `wonderful variety' refers to a smooth, connected, complete variety with a group action such that there is an open orbit and whose boundary divisors have normal crossings and satisfy additional properties. For more details, see \cite{Luna}.} of $\Gadbar$.
\end{enumerate}

\begin{example} The wonderful compactification of $\PSL_2$ is $ \C\P^3$. However, for $n \geq 3$, the wonderful compactification of $\PSL_n$ is not $\C\P^{n^2-1}$, and is somewhat more difficult to describe explicitly.  \end{example}

The wonderful compactification captures the asymptotics of matrix coefficients in the following way. The Peter-Weyl theorem asserts that the coordinate algebra $\O(G)$ is spanned by matrix coefficients of representations of $G$. The fact that the isomorphism classes of irreducible finite-dimensional representations of $G$ are labeled by dominant weights leads to the definition of a (multi-)filtration on $\O(G)$ by the weight lattice $\Lambda$ of $G$. We refer to this filtration as the Peter-Weyl filtration on $\O(G)$.

\begin{theorem}\label{thm:Brion} \cite[Theorems 2.2.3 and 3.2.3]{Brion} The following $\Lambda$-graded algebras are isomorphic:
\begin{enumerate}
 \item\label{item:rees} The (multi-)Rees algebra of $\O(G)$ for the Peter-Weyl filtration.
 \item\label{item:cox} The total coordinate ring, or Cox ring, of the wonderful compactification $\Gadbar$. In particular, the Picard group of $\Gadbar$ is identified with $\Lambda$.
 \item The coordinate ring $\O(\VinbG)$ of the Vinberg semigroup $\Vinb_G$ of $G$. 
\end{enumerate}
\end{theorem}

We recall the definition of the Vinberg semigroup in Section \ref{sec:classical}. For now, one can think of $\VinbG$ as the affine variety associated to either of the isomorphic rings in (\ref{item:rees}) or (\ref{item:cox}) of Theorem \ref{thm:Brion}. As the name indicates, the variety $\VinbG$ carries a canonical semigroup structure, and the group of units is, up to a finite group, the direct product  of $G$ with a maximal torus $T$. The weight lattice $\Lambda$ of $G$ can be realized as the character lattice of $T$, and so the $\Lambda$-grading on $\O(\VinbG)$ corresponds to an action of $T$ on $\VinbG$.

\begin{theorem}\cite[Theorem 5.3]{MartensThaddeus}\label{thm:GIT} Let $\lambda$ be a regular dominant weight, regarded as a character of $T$. The wonderful compactification  is isomorphic to the geometric invariant theory (GIT) quotient of $\Vinb_G$ by $T$ along $\lambda$: $$\Gbarad = \Vinb_G \text{ $\!$/$\! \!$/$\!$}_\lambda  T.$$ \end{theorem}

\begin{rmk} The wonderful compactification admits a realization as a moduli space of certain framed bundle chains, as demonstrated by Martens and Thaddeus \cite{MartensThaddeus}. This perspective precipitates the construction, for any given semisimple group, of a distinguished smooth stack that compactifies the group. If the group has trivial center, this stack coincides with the wonderful compactification. \end{rmk}

\subsection{What is the wonderful compactification for quantum groups?}\label{subsec:whatis} A key tenet of algebraic  geometry, due to Grothendieck, asserts that a space can be completely understood through its category of sheaves. Furthermore, $q$-deformations of the category of sheaves can be viewed as categories of sheaves on a (nonexistent) quantum version of the original space. Thus, replacing a space with a category provides more flexibility in producing deformations; this is starting point of much of noncommutative geometry. 

Applying this philosophy to the case of the wonderful compactification, we seek a category $\QCoh_q(\Gadbar)$ that forms a $q$-deformation of the category $\QCoh(\Gadbar)$ of quasicoherent sheaves on the wonderful compactification $\Gadbar$. Here $q$ is a nonzero complex number that is not a root of unity. 

To obtain the desired $q$-deformation, we take inspiration from a result of Serre, which describes quasicoherent sheaves on a projective variety in terms of graded modules for the homogeneous coordinate ring. Specifically, let $A = \bigoplus_{n \in \Z} A_n$ be commutative graded ring, and let $X$ be the associated projective scheme with twisting sheaf $\O(1)$. A graded $A$-module is called torsion if every element is annihilated by $A_{\geq N}$ for some $N$. The quotient of the category of graded modules for $A$  by the full subcategory of torsion modules produces an abelian category denoted $\uProj(A)$. There is a functor of graded global sections:
\begin{align*} \Gamma_* : \QCoh(X) &\rightarrow \uProj(A) \\
\cF &\mapsto  \bigoplus_{n \in \Z} \Gamma(X, \cF \otimes \O(1)^{\otimes n}).\end{align*} 

\begin{theorem}\cite{Serre} If $A$ is finitely generated by elements of degree one over a field, then the functor $\Gamma_*$ of graded global sections is an equivalence of categories.
\end{theorem}

The definition of $\uProj(A)$ works just as well when $A$ is a noncommutative graded ring. Moreover, rather than just considering rings graded by the integers, one can consider rings graded by the weight lattice $\Lambda$ of a semisimple group $G$. For such a ring $R$, one can make sense of a torsion graded $R$-module, and form the quotient $\uProj(R)$ of the category of graded $R$-modules by the subcategory of graded torsion modules.

We now outline the construction of the quantum wonderful compactification:

\begin{enumerate} 
 \item Use Theorem \ref{thm:Brion} to establish an equivalence between the category $\QCoh(\Gadbar)$ of quasicoherent sheaves on the wonderful compactification and the category $\uProj(\O(\VinbG))$, where $\O(\VinbG)$ is the $\Lambda$-graded algebra of functions on the Vinberg semigroup. 
 \item Produce a $q$-deformation $\O_q(\VinbG)$ of $\O(\VinbG)$, compatible with relevant structures. 
 \item Define the category of quasicoherent sheaves on the quantum wonderful compactification as the category $\QCoh_q(\Gadbar) := \uProj\left(\O_q(\VinbG)\right)$. 
\end{enumerate}

\subsection{Main results}\label{subsec:mainresults}

Fix $q$ to be a nonzero complex number that is not a root of unity. A version of the Peter-Weyl theorem holds for the quantum coordinate algebra $\O_q(G)$ and we obtain a filtration on $\OqG$ by $\Lambda$, which we refer to as the Peter-Weyl filtration.

\begin{prop}[Proposition \ref{prop:qvinberg}] The (multi-)Rees algebra of $\OqG$ for the Peter-Weyl filtration is a $q$-deformation of the coordinate ring of the Vinberg semigroup $\Vinb_G$ of $G$, and quantizes a certain Poisson structure on $\Vinb_G$. We denote this Rees algebra by $\O_q(\Vinb_G)$. \end{prop}

Observe that the algebra $\O_q(\Vinb_G)$ carries a grading by $\Lambda$. As it is a $q$-deformation of the total coordinate ring of the wonderful compactification, we make the following definition, which recovers the category of quasicoherent sheaves on $\Gbarad$ when $q =1$.

\begin{definition}[Definition \ref{def:maindefs}]  The category of quasicoherent sheaves on the quantum wonderful compactification is defined as $$\QCoh_q\left(\Gbarad\right) = \uProj\left( \O_q(\Vinb_G) \right).$$ \end{definition}

Recall that the variety $\Gadbar$ is stratified by the orbits of a $G \times G$ action, and these orbits are indexed by subsets of the set of positive simple roots. Given such a subset $I$, we write $\Orb_I$ for the corresponding orbit, and $\Lambda_I$ for the sublattice of $\Lambda$ spanned by the roots in $I$. There is a filtration on $\OqG$ by the quotient lattice $\Lambda/\Lambda_I$, which is a coarser filtration than the Peter-Weyl filtration. Let $\gr_I(\OqG)$ denote the associated graded algebra. 

\begin{definition}[Definition \ref{def:qorbits}] Fix a subset $I$ of positive simple roots. The category of quasicoherent sheaves on the quantum orbit corresponding to $I$ is defined as $$\QCoh_q\left(\Orb_I\right) = \uProj\left( \gr_I(\OqG) \ot \C[\Lambda_I] \right).$$ \end{definition}

A subset $I$ of positive simple roots determines a parabolic subgroup $P=P_I$ of $G$ and an opposite parabolic $P^-$. These have a common Levi group $L$, whose Lie algebra is denoted $\frakl$. Let $\fraku$ and $\fraku^-$ be the Lie algebras of the unipotent radicals of $P$ and $P^-$. 

\begin{theorem}[Theorem \ref{thm:gradedinvariants}] \label{thm:qorbitsintro}There is an isomorphism of $\Lambda/\Lambda_I$-graded algebras
$$\gr_I(\OqG)= \O_q(G \times G)^{\uqulu}$$
\end{theorem}

Consequently, the quantum orbits can also be described in terms of certain graded subalgebras of $\O_q(G\times G)$, which are given as invariants for a quantum group action. When $I$ consists of all positive simple roots, the subalgebra in question is the `diagonal' copy of $\O_q(G)$. At the other extreme, when $I = \emptyset$, the multihomogeneous coordinate ring of the classical orbit $G/B \times B^-\backslash G$ is the algebra of functions on the asymptotic cone $(G/N \times N^-\backslash G)/T$, which quantizes precisely to $\O_q(G \times G)^{U_q(\mathfrak{n} \times \frakt \times \mathfrak{n}^-)}$. By Theorem \ref{thm:qorbitsintro}, this algebra is isomorphic to the full associated graded algebra $\gr(\OqG)$ for the Peter-Weyl filtration.

\subsection{Motivation and future directions}\label{subsec:motivation}

This paper forms the first step in a program to understand the role of the wonderful compactification in quantum geometric representation theory. We remark on several goals in this program that provide motivation for the current work.

In joint on-going work with D.~Ben-Zvi and D.~Nadler, we aim to place the Beilinson-Bernstein localization theorem within the framework of the wonderful compactification and asymptotics of matrix coefficients. Precursors to our work appear in work of Ben-Zvi and Nadler, and of Emerton, Nadler, and Vilonen \cite{BZNHarishChandra, ENV}. We expect some of our techniques, together with the newly-introduced quantum wonderful compactification, to place the quantum Beilinson-Bernstein theorem within the same framework of asymptotics of matrix coefficients. This approach requires the development of the notion of quantum differential operators on the Vinberg semigroup, and an analogue  of the Verdier specialization functor for quantum $\D$-modules  

Another source of motivation is the development of the theory of quantum character sheaves as a contribution to the study of harmonic analysis on quantum groups. The theory of character sheaves for quantum groups will involve:
\begin{enumerate}
 \item The construction of the appropriate $q$-deformation of the Hecke category of Borel equivariant $\D$-modules on the flag variety. Related constructions appear in work of Backelin and Kremnitzer on quantum flag varieties \cite{BaKr}.
 \item The elevation of various links between character sheaves and the wonderful compactification (see, e.g.\ \cite{BFO}) to the quantum level.
 
 
 \item An alignment with the quantum geometric Langlands program. In particular, we expect direct relations with topological field theories arising from quantum groups and their categories of representations \cite{BBJ}. 
\end{enumerate}

Another future direction is the consideration of the case when the quantum parameter $q = \epsilon$ is a root of unity. In this case, the quantum coordinate algebra $\O_\epsilon(G)$ is finite dimensional over its center, and contains the unquantized coordinate algebra $\O(G)$ as a central sub-Hopf algebra. It is reasonable to expect that the category $\QCoh_\epsilon(\Gadbar)$ to be the category of modules for a certain sheaf of algebras $\mathcal A$ on $\Gadbar$, and that $\mathcal A$ has an Azumaya locus related to double Bruhat cells.

\subsection{Outline}\label{subsec:outline} We now describe the contents of this paper. Section \ref{sec:preliminaries} contains preliminary material, including background on matrix coefficients for Hopf algebras (Section \ref{subsec:matrixcoef}) and notation for algebraic groups and quantum groups (Section \ref{subsec:notation}). Section \ref{subsec:projbasic} introduces the formalism of $\uProj$ categories for noncommutative, multi-graded algebras, including the notion of torsion graded modules, compatible Poisson structures, and the relation to quasicoherent sheaves. 

Section \ref{sec:classical} gives an expository account of the construction of the Vinberg semigroup (Section \ref{subsec:vinberg}) and wonderful compactification (Section \ref{subsec:wonderful}). Although there are no new results in Section \ref{sec:classical}, our approach is somewhat more algebraic than that of other authors, and provides insight on the quantum case. We describe the stratification of the wonderful compactification by $G \times G$ orbits in Section \ref{subsec:orbits}, explain Poisson structures on the Vinberg semigroup and wonderful compactification in Section \ref{subsec:poissonvinb}, and examine the case of $\SL_2$ in Section \ref{subsec:sl2}. 

Section \ref{sec:quantumwc} forms the heart of this paper. We begin the section with the construction of the quantum Vinberg semigroup and the quantum wonderful compactification (Section \ref{subsec:qdefs}).  We introduce filtrations on the quantum coordinate algebra $\O_q(G)$ in Section \ref{subsec:qfiltrations} and explain the connection to quantum flag varieties. Finally, we use these filtrations to describe the quantum orbits in Section \ref{subsec:qorbits}, which is perhaps the most technical section of this paper. 

Section \ref{sec:qsl2} examines the general constructions of Section \ref{sec:quantumwc} in the case of $G = \SL_2$. We include basic background (Section \ref{subsec:uosl2}) followed by a description of the Peter-Weyl filtration for $\O_q(\SL_2)$ (Section \ref{subsec:pwsl2}). We state results on the quantum Vinberg semigroup and wonderful compactification for $\SL_2$ in Section \ref{subsec:vinbergsl2}.

\subsection{Acknowledgements} This paper is based on material from the author's PhD thesis. He would like to warmly thank his advisor, David Ben-Zvi, for suggesting the topic of this paper, and, more importantly, for his constant encouragement, patience, and guidance throughout graduate school. Special thanks goes to David Jordan, Kobi Kremnitzer, Travis Schedler, Pavel Safronov, and Sam Gunningham for numerous conversations. In addition, the author would like to thank  Iain Gordon, Anthony Henderson, Masoud Kamgarpour, Sean Keel, Johan Martens, and Ben Webster for insightful questions and discussions.  This work was supported by the National Science Foundation through a Graduate Research Fellowship, and through grant No.\ 0932078000 while the author was in residence at the Mathematical Sciences Research Institute during Fall 2014.

\section{Preliminaries}\label{sec:preliminaries}

This section collects background, notation, and other preliminary material that will be used in subsequent sections. The reader may skip this section on first reading and refer to it as necessary. Throughout the remainder of this paper, unless specified otherwise, the ground field is $\C$.

\subsection{Matrix coefficients}\label{subsec:matrixcoef}

The following discussion follows Section 1.9 of \cite{BrownGoodearl}.

\begin{definition} The Hopf dual of an algebra $A$ is defined as $$A^\circ = \{  f\in A^* \ | \ f(I)=0 \ \text{for some ideal $I$ of $A$ with $\dim(A/I) < \infty$}\}.$$  \end{definition}

\begin{lemma} Let $A$ be an algebra with multiplication $m$ and unit $\eta$. The Hopf dual $A^\circ$ is a coalgebra with $\Delta = m^*$ and $\epsilon = \eta^*$. Moreover, if $H = (H, m, \eta, \Delta, \epsilon)$ is a Hopf algebra, then $H^\circ = (H^\circ, \Delta^*, \epsilon^*, m^*, \eta^*)$ is a Hopf algebra. \end{lemma}

\begin{definition} Let $M$ be a (left) module over $H$. For $v \in M$ and $f \in M^*$ define the coordinate function $c_{f,v}^M \in H^*$ as  $c_{f,v}^M(h) = f(hv)$ for $h \in H$. Thus, we have a map $c^M : M^* \otimes M \rightarrow H^*$ taking $ f \otimes v$ to $c_{f,v}^M$. The image of $c^M$ is called the set of `matrix coefficients' for $M$. 
\end{definition}

\begin{lemma}\label{lem:matrixcoeff} We collect the following basic properties of matrix coefficients:
\begin{enumerate}
\item If $M$ is finite-dimensional, then its matrix coefficients lie in $H^\circ$. 

\item The map $c^M: M^* \otimes M \rightarrow H^*$ is $H \times H$-equivariant. Consequently, if $M$ is an irreducible $H$-module, then $c^M$ is injective.

\item \label{lem:matrixcoeffbasics} Let $M$ and $N$ be finite dimensional modules for $H$, let $v \in M$, $f \in M^*$, $w \in N$ and $g \in N^*$, and let  $v_i$ and $f_i$ be dual bases of $M$ and $M^*$. Then:
$$c_{f,v}^M + c_{g,w}^N = c_{(f,g), (v,w)}^{M \oplus N}  \qquad c_{f,v}^M \cdot c_{g,w}^N = c_{f\otimes g, v \otimes w}^{M \otimes N} $$ $$ \Delta (c^{M}_{f,v}) = \sum_i c^{M}_{f, v_i} \otimes c^{M}_{f_i,v} \qquad \epsilon(c_{f,v}^M) = f(v) \qquad S(c_{f,v}^M) = c_{v,f}^{M^*}.$$ In particular, the coproduct on $H^\circ$ sends the image of $c^M$ to that of $c^M \ot c^M$.

\item \label{lem:subbialgebra} As an $H\times H$-module, $H^\circ$ is isomorphic to the directed union of  the matrix coefficients for finite-dimensional irreducible $H$-modules $M$:
$$\bigoplus_\text{\rm $M$ fin. dim. irr.} M^* \otimes M \stackrel{\sim}{\longrightarrow} H^\circ.$$

\item \label{lem:matrixcoeffmorphism} Suppose $\phi: M \rightarrow N$ is an $H$-equivariant homomorphism, and let $\phi^* :N^* \rightarrow M^*$ be the dual homomorphism. Then $c_{f,\phi(m)}^N = c_{\phi^*(f), m}^M$ for any $m  \in M$ and $f \in N^*$. 
\end{enumerate} \end{lemma}

Let  $F$ be a family of finite-dimensional $H$-modules, and let $\hat F$ denote the closure of $F$ under finite direct sums and tensor products. 

\begin{lemma}\label{lem:family}  Let $A$ be the subalgebra of $H^\circ$ generated by all matrix coefficients of elements in $F$. Then $A$ is a sub-bialgebra of $H^\circ$, and, as an $H\times H$-module, is the directed union of  the spaces of matrix coefficients for $M \in \hat F$. Moreover, if $F$ is closed under duals, then $A$ is a sub-Hopf algebra of $H^\circ$. \end{lemma}

\subsection{Algebraic groups and quantum groups}\label{subsec:notation}

Let $G$ be a connected semisimple algebraic group over $\C$ with Lie algebra $\g$. Fix a Borel subgroup $B \subseteq G$ and a maximal torus $T \subseteq B$. Write $\mathfrak b$ and $\frakt$ for the corresponding Lie subalgebras of $\g$. The Borel subgroup $B$ has unipotent radical $N$, with Lie algebra $\mathfrak{n}$, and it has an opposite Borel subgroup $B^-$ uniquely characterized by the property that $B \cap B^- = T$. Let $r$ be the rank of $G$. Write $Z = Z(G)$ for the center of $G$, and $\Gad = G/Z(G)$ for the adjoint group of $G$. 

The weight lattice $\Lambda_W$ of $\g$ is generated by the fundamental weights $\omega_1, \dots, \omega_r$. The weight lattice contains the cone $\Lambda_W^+$ of dominant weights. The interior of $\Lambda_W^+$ is the set of regular dominant weights. Thus, dominant weights comprise the nonnegative linear combinations of the fundamental weights, and regular dominant weights comprise the positive linear combinations of fundamental weights.  Fix a set of positive simple roots $\{ \alpha_1, \dots, \alpha_r\}$ of $T$ relative to $B$. These generate the root lattice $\Lambda_R$, and we use the set $\Delta = \{1, \dots, r\}$ to index the positive simple roots. 

\begin{definition}\label{def:gorder} Define a partial order on $\Lambda_W$ by setting $\mu\leq \lambda$ whenever $\lambda-\mu$ is a nonnegative multiple of positive simple roots. Similarly, we write $\lambda < \mu$ if $\lambda \leq \mu$ and $\lambda \neq \mu$. This partial order is referred to as the dominance ordering on the weight lattice $\Lambda_W$ \end{definition}

The weight lattice $\Lambda_G$ of $G$ is the character lattice $X^*(T)$ of the maximal torus. We have inclusions of lattices: $\Lambda_R \subseteq \Lambda_G \subseteq \Lambda_W$. The set of isomorphism classes of finite-dimensional irreducible representations of $G$ are in bijection with points in the cone $\Lambda_G^+ = \Lambda_W^+ \cap \Lambda_G$ of dominant weights for $G$. We denote by $V_\lambda$ the irreducible representation corresponding to $\lambda \in \Lambda_G^+$. Points in the interior of $\Lambda_G^+$ are called regular dominant weights for $G$. In a context where the group $G$ is fixed, we write $\Lambda$ instead of $\Lambda_G$.

\begin{lemma}\label{lem:tensorproducts} If $V_\nu$ appears as an irreducible subrepresentation of the tensor product $V_\lambda \otimes V_\mu$, then $\nu \leq \lambda+ \mu$ for the dominance ordering. \end{lemma}

Given a subset $I \subseteq \Delta$, denote by $P_I$ the parabolic subgroup of $G$ whose Lie algebra $\mathfrak{p}_I$ is generated by $\mathfrak b$ and the root vectors corresponding to the roots $-\alpha_i$ for $i \in I$. Let $U_I$ be the unipotent radical of $P_I$ (with Lie algebra $\mathfrak{u}_I$), and let $L_I$ denote the subgroup of $P_I$ whose Lie algebra $\mathfrak{l}_I$ is generated by $\mathfrak t$ and the root vectors corresponding to the roots $\pm \alpha_i$ for $i \in I$. Then $L_I$ is a maximal reductive subgroup of $P_I$, and is called a Levi subgroup of $G$. We have $P_I = U_I \rtimes L_I$. Similarly, we define the opposite parabolic $P_I^-$ and its unipotent radical $U_I^-$. Observe that $P_I \cap P_I^- = L_I$. Write $\mathfrak{p}_I^-$ and $\mathfrak{u}_I^-$ for the corresponding Lie algebras.



We will be interested in the so-called standard Lie bialgebra structure on $\g$, described in \cite[Example 1.3.8]{ChariPressley}. It is a coboundary Lie bialgebra structure determined by a certain skew-symmetric element $r \in \bigwedge^2 \g$ that satisfies the classical Yang-Baxter equation. Write $\O(G)$ for the coordinate algebra of $G$. For $f \in \O(G)$, we write  $X^L f$ and $X^R f$ for the action of $X \in \g$ on $f$ by left- and right-invariant vector fields, respectively. 

\begin{prop}{\cite[Theorem 1.3.2, Section 2.2]{ChariPressley}}\label{prop:poissonlie}  Write $r$ $=$ $\sum_{i} a_i \otimes b_i$ for the element that determines for the standard Lie bialgebra structure on $\g$. There is a Poisson-Lie  bracket on $\O(G)$  given by  
 $$ \{ f_1, f_2\} = \sum_{i} ((a_i^L f_1) ( b_i^L f_2) - (a_i^R f_1) (b_i^R f_2)).$$ \end{prop}

\begin{rmk} The same formula gives a Poisson-Lie structure on $G$ for any $r \in \bigwedge^2 \g$ that defines a coboundary Lie bialgebra structure on $\g$. We note that when $\g$ is simple, all Lie bialgebra structures are coboundary \cite[Example 2.1.7]{ChariPressley}.  \end{rmk}

The standard Lie bialgebra structure on $\g$ admits a quantization, which leads to the quantized enveloping algebra $\uqg$. For the definition of $\uqg$, we follow \cite[Chapter I.6]{BrownGoodearl} and \cite[Chapter 6]{KlimykSchmudgen}. Fix $q \in \C^\times$. We will assume throughout that $q$ is not a root of unity.  Let $C = (a_{ij})$ be the Cartan matrix of $\g$. For $i= 1, \dots, r$, set $d_i = (\alpha_i, \alpha_i)/2$, where $(,) : \g \times \g \rightarrow \C$ is the Killing form, and set $q_i = q^{d_i}$. 

\begin{definition} The Drinfeld-Jimbo quantized enveloping algebra $\uqg$ of $\g$ is defined as the algebra generated by elements $E_1, \dots, E_r, F_1, \dots, F_r, K_1^{\pm 1}, \dots, K_r^{\pm 1}$, with relations 
$$ K_i K_j = K_j K_i, \quad K_i E_j K_i\inv = q_i^{a_{ij}} E_j, \quad K_i F_j K_i\inv = q_i^{-a_{ij}} F_j, \quad E_i F_j - F_j E_i = \delta_{ij} \frac{K_i - K_i\inv}{q_i - q_i\inv},$$
and the quantum Serre relations, which we omit here. \end{definition}

\begin{prop}\label{prop:uqghopfstr} The algebra $\uqg$ has the following Hopf algebra structure:
\begin{align*}
\Delta(K_i) &= K_i \otimes K_i &\qquad \epsilon(K_i) &=1 & \qquad S(K_i) &= K_i\inv \\
\Delta(E_i) &= E_i \otimes 1 +  K_i \otimes E_i  &\qquad \epsilon(E_i) &=0 & \qquad S(E_i) &= -K_i\inv E_i \\
\Delta(F_i) &= F_i \otimes K_i\inv  + 1 \otimes F_i &\qquad \epsilon(F_i) &=0 & \qquad S(F_i) &= -F_i K_i 
\end{align*} \end{prop}

\begin{definition} We define the following subalgebras of $U_q(\g)$:
 \begin{itemize}
  \item $U_q(\mathfrak{p}_I)$ is generated by $K_i^{\pm 1}$ and $E_i$ for $i = 1, \dots, r$, as well as $F_i$ for $i \in I$. 
  \item  $U_q(\mathfrak{p}_I^-)$ is generated by $K_i^{\pm 1}$ and $F_i$ for $i = 1, \dots, r$, as well as $E_i$ for $i \in I$.  
  \item  $U_q(\mathfrak{u}_I)$ is generated by $E_i$ for $i \notin I$.
  \item  $U_q(\mathfrak{u}_I^-)$ is generated by $F_i$ for $i \notin I$. 
  \item  $U_q(\mathfrak{l}_I)$ is generated by $K_i^{\pm 1}$ for $i = 1, \dots, r$, as well as $E_i$ and  $F_i$ for $i \in I$.  
 \end{itemize}  \end{definition}

We will be interested exclusively in type $\mathbf 1$ representations of $\uqg$. Let  $\mathcal C_q(\g)$ denote the category of finite-dimensional $\uqg$-modules of type $\mathbf 1$. 

\begin{prop} The category $\mathcal C_q(\g)$ is a semisimple rigid tensor subcategory of $\uqg\dmod$ whose irreducible objects are in bijection with the set $\Lambda_W^+$ of dominant weights of $\g$. \end{prop}

For $\lambda \in \Lambda_W^+$, we write ${\mathcal V}_\lambda$ for the corresponding irreducible representation of $\uqg$. Let $\mathcal C_q(G)$ denote the subcategory of $\mathcal C_q(\g)$ generated under finite direct sums and tensor products by the irreducible representations whose highest weights lie in $\Lambda_G^+$. Using Lemma \ref{lem:family}, we make the following definition.

\begin{definition}\label{def:oqg} The quantized coordinate algebra $\OqG$ of $G$ is defined as sub-Hopf algebra of $\uqg^\circ$ generated by the matrix coefficients of all modules in $\mathcal C_q(G)$. \end{definition}

The algebra $\OqG$ is a flat deformation of the algebra $\OG$, and quantizes the Poisson-Lie structure on $\OG$ arising from the standard Lie bialgebra structure on $\g$.

\begin{lemma}\label{lem:qtensorproducts} If $\mathcal V_\nu$ appears as an irreducible subrepresentation of the tensor product $\mathcal V_\lambda \otimes \mathcal V_\mu$, then $\nu \leq \lambda+ \mu$ in the dominance ordering on $\Lambda_W$. \end{lemma}

\subsection{$\Lambda$-graded algebras}\label{subsec:projbasic}

This section collects definitions and results concerning $\uProj$ categories for noncommutative rings graded by a lattice $\Lambda$. We present reformulations and refinements of constructions that appear in work of Artin and Zhang, and of Ginzburg, and are ultimately inspired by results of Serre \cite{ArtinZhang, Ginzburg, Serre}.

Let $\Lambda$ be a lattice, that is, a finitely generated torsion-free abelian group. Let $\Lambda^+$ be a subsemigroup of $\Lambda$ (i.e.\ a cone). We assume that $\Lambda^+$ has the following properties:

\begin{enumerate}
 \item For any $\lambda_1$ and $\lambda_2$ in $\Lambda$, the intersection $(\lambda_1 + \Lambda^+) \cap (\lambda_2 + \Lambda^+)$ is nonempty. 
 \item $\Lambda^+ \cap ( - \Lambda^+) = \{ 0\}$.
\end{enumerate}

\begin{definition}\label{def:lambdagraded} A $\Lambda$-graded algebra is a $\Lambda$-graded vector space $R = \bigoplus_{\lambda \in \Lambda} R_\lambda$ over $\C$ equipped with an associative $\C$-linear multiplication map $R \ot R \rightarrow R$ that restricts to a map $R_\lambda \otimes R_\mu \rightarrow R_{\lambda + \mu}$ for any $\lambda, \mu \in \Lambda$. 
\end{definition}

When $\Lambda = \Z$ and $\Lambda^+ = \Z_{\geq 0}$, we recover the usual notion of a graded algebra. The setting we will be most interested in is when $\Lambda = X^*(T)$ is the weight lattice of $G$ with respect to a maximal torus $T$, and $\Lambda^+$ is the cone of dominant weights.

\begin{definition} Let $R$ be a $\Lambda$-graded algebra. 
\begin{enumerate}
\item A graded left $R$-module is a $\Lambda$-graded vector space $M = \bigoplus_{\lambda \in \Lambda} M_\lambda$ equipped with an action of $R$ such the action map restricts to a map $R_\lambda \otimes M_\mu \rightarrow M_{\lambda + \mu}.$ The category of graded left $R$-modules is denoted $\Grmod(R)$ and its objects will henceforth be referred to simply as $R$-modules. 

\item An $R$-module $M$ is finitely generated if there exist elements $m_1 \in M_{\nu_1}, \dots, m_p \in M_{\nu_p}$, where $\nu_i \in \Lambda$, such that the map $$\bigoplus_{i=1}^p R_{\lambda - \nu_i} \rightarrow M_\lambda \ ; \qquad (r_i) \mapsto \sum_{i=1}^p r_i m_i$$ is surjective for all $\lambda \in \Lambda$.  The category of finitely generated graded left $R$-modules is denoted $\grmod(R)$, and it is a full subcategory of $\Grmod(R)$. 
\end{enumerate} \end{definition}

\begin{rmk} We will mostly be interested in the case where $R$ is noetherian and locally finite, i.e.\ $R_\lambda$ is finite-dimensional over $\C$ for every $\lambda \in \Lambda$. When $R$ is noetherian, the category $\grmod(R)$ is abelian. \end{rmk}

\begin{definition}
An $R$-module $M$ is called torsion if, for all $m$ in $M$, there exists  $\lambda \in \Lambda^+$ such that $R_\mu$ acts by zero on $m$ for any $\mu \in  \lambda + \Lambda^+$. The full subcategory of torsion modules (resp. finitely generated torsion modules) is denoted $\Tors(R)$ (resp. $\tors(R)$). \end{definition}

\begin{lemma}\label{lem:fgzero} If $M$ is finitely generated, then $M$ is torsion if and only if there exists $\lambda \in \Lambda$ such that  $M_\mu = 0$ for $\mu \in \lambda + \Lambda^+$.  \end{lemma}

\begin{proof} Let  $m_1 \in M_{\nu_1}, \dots, m_p \in M_{\nu_p}$ be generators of $M$. If $M$ is torsion, then for each $m_i$, there exists $\lambda_i \in \Lambda^+$ such that $R_\mu$ acts by zero on $m_i$ for $\mu \in \lambda_i + \Lambda^+$. Choose $\lambda \in \bigcap_i (\nu_i + \lambda_i + \Lambda^+)$. For every $\mu \in \lambda + \Lambda^+$, we have $\mu \in \nu_i + \Lambda^+$ for all $i$. Therefore, the map  $ \bigoplus_{i=1 }^p R_{\mu - \nu_i} \rightarrow M_\mu$ taking $(r_i)$ to $\sum_{i=1}^p r_i m_i $ is surjective. On the other hand, $\mu - \nu _i \in \lambda - \nu_i + \Lambda^+ \subseteq \lambda_i + \Lambda^+$ for all $i$, so $r_i m_i = 0$ for all $i$. We conclude that $M_\mu = 0$ for all $\mu \in \lambda + \Lambda^+ .$ Conversely, suppose that there exists $\mu$ such that $M_\mu = 0$ for $\mu \in \lambda + \Lambda^+$. Let $m\in M_\nu$. Let $\lambda' \in (\nu + \Lambda^+) \cap  (\lambda + \Lambda^+)\subseteq \Lambda^+$. Then $R_{\mu} m = 0$ for all $\mu \in \lambda' - \nu + \Lambda^+$. \end{proof}

\begin{definition} A full subcategory $\mathcal T$ of an abelian category $\mathcal A$ is called dense\footnote{Dense subcategories are also referred to as Serre subcategories.}  if it is closed under extensions. In other words, for any short exact sequence $$0 \rightarrow M' \rightarrow M \rightarrow M'' \rightarrow 0$$ of objects in $\mathcal A$, the object $M$ belongs to $\mathcal T$ if and only if $M'$ and $M''$ both belong to $\mathcal T$. \end{definition}

\begin{lemma} Suppose $R$ is noetherian. Then the full subcategory of torsion objects in either of $\grmod(R)$ or $\Grmod(R)$ is dense. \end{lemma}

\begin{proof} Fix a short exact sequence $0 \rightarrow M' \rightarrow M \rightarrow M'' \rightarrow 0$ in $\Grmod(R)$. If $M$ is torsion, then it is clear that $M'$ and $M''$ are both torsion. For the other direction, suppose $M'$ and $M''$ are torsion and let $m$ be a homogeneous element of $M$. It is enough to assume that $m$ is of degree zero and $M$ is generated by $m$, and thus we reduce to the setting of $\grmod(R)$. In this case, there are graded right ideals $I \subseteq J \subseteq R$ such that $M = R/I$, $M'' = R/J$, and $M' = J/I$. Since the image of $m$ in $M''$ is torsion, there exists $\lambda_1 \in \Lambda^+$ such that $R_{\lambda_1} m \subseteq M'$. Since $M'$ is a finitely generated torsion module, there exists $\lambda_2$ such that $M'_{\mu} = 0$ for $\mu \in \lambda_2 + \Lambda^+$. Let $\lambda \in (\lambda _1 + \Lambda^+) \cap ( \lambda_2 + \Lambda^+)$. Then, for any $\mu \in \lambda + \Lambda^+$, $R_\mu m \subset M'_\mu = 0$. 
\end{proof}

General results on the localization of abelian categories (see \cite[Section 4.3 and 4.4]{Popescu}) allow us to make the following definition, and justify the subsequent proposition. 

\begin{definition} Suppose $R$ is noetherian. We form  the quotient categories $\uProj(R) =$ $\Grmod(R)/ \Tors(R)$ and $\uproj(R)$ $=$ $\grmod(R)/ \tors(R)$. \end{definition}

\begin{prop} The category  $\uProj(R)$ is abelian. If $R$ is noetherian, then the category $\uproj(R)$ is abelian and noetherian. \end{prop}

We conclude this section with a discussion of the connection between $\uProj$ categories and quasicoherent sheaves on projective varieties. For the remainder of this section, suppose $R$ is a commutative $\Lambda$-graded algebra, and suppose $\Lambda = X^*(T)$ is the character lattice of a torus $T$. Thus, $\Spec(R)$ is a scheme with an action of the torus $T$. 

\begin{definition} The GIT quotient of $\Spec(R)$ by $T$ at the character $\lambda \in \Lambda$, denoted $\Spec(R) \text{ $\!$/$\! \!$/$\!$}_\lambda  T$, is the projective scheme associated with the $\Z_{\geq 0}$-graded algebra $\bigoplus_{n \geq 0} R_{n\lambda}$: 
$$\Spec(R) \text{ $\!$/$\! \!$/$\!$}_\lambda  T = \Proj\left( \bigoplus_{n \geq 0} R_{n\lambda} \right).$$
\end{definition}

\begin{notation} Let $A$ be a $\Z_{\geq 0}$-graded ring. We write $\Proj(A)$ for the projective scheme associated to $A$, and $\uProj(A)$ for the category of graded $A$-modules modulo torsion modules. In other words, roman font indicates a space, while sans serif font indicates a category.  \end{notation}

\begin{prop}\label{prop:qcohproj} Suppose the GIT quotients $ \Spec(R) \GIT_\lambda T$ and $ \Spec(R) \GIT_\mu T$ coincide for $\lambda, \mu$ in the interior of $\Lambda^+$. Then, for $\lambda $ in the interior of $\Lambda^+$, there is an equivalence of categories $$\QCoh( \Spec(R) \GIT_\lambda T) \simeq \uProj( R).$$ \end{prop}

\begin{proof}[Sketch of proof.] Serre's theorem (Section \ref{subsec:whatis0}) implies that the categories $\QCoh( \Spec(R) \GIT_\lambda T)$ and $\uProj( \bigoplus_{n \geq 0} R_{n \lambda})$ are equivalent. The category $\uProj( R)$ has a shift functor $M \mapsto M[\lambda]$, which is ample in the sense of Artin and Zhang \cite[Section 4]{ArtinZhang}. Theorem 4.5 of Artin and Zhang implies that there is an equivalence of categories $\uProj( \bigoplus_{n \geq 0} R_{n \lambda}) \simeq \uProj(R).$ \end{proof}

\begin{example} Note that the quotient $G/N$ carries a residual action of $T$, so its algebra of global functions $\O(G/N)$ is graded by the weight lattice $\Lambda_G = X^*(T)$. Moreover, the GIT quotient of the affine closure $\Spec(\O(G/N))$ of $G/N$ by $T$ along a regular character $\lambda$ coincides with the flag variety $G/B$. By the preceding proposition, we have an equivalence of categories: $\QCoh(G/B) = \uProj(\O(G/N)).$ \end{example}

\begin{definition} We say that a Poisson bracket $\{, \}: R \otimes R \rightarrow R$ on $R$ is compatible with the $\Lambda$-grading if it restricts to a map $\{,\} : R_\lambda \otimes R_\mu \rightarrow R_{\lambda + \mu}.$
\end{definition}

\begin{lemma}\label{lem:poissonGIT} A Poisson bracket on $R$ that is compatible with the $\Lambda$-grading descends to a Poisson structure on the GIT quotient $\Spec(R)\GIT_\lambda  T$ for any $\lambda \in \Lambda$. \end{lemma}

\begin{proof} Let $X = \Spec(R)\GIT_\lambda T$ denote the GIT quotient. Then $X$ admits an open cover by affine schemes of the form $\Spec(R[r\inv]_0)$ where $r \in R_{n\lambda}$ for some $n\geq 1$, and $R[r\inv]_0$ denotes the degree 0 part of the $\Lambda$-graded algebra $R[r\inv]$. Fix such an $r$ and define a Poisson bracket on $R[r\inv]$ by setting $\{a, r\inv\} = -r^{-2} \{a, r\}$ for $a \in R$, and extending by linearity and skew-symmetry. This bracket restricts to a Poisson bracket on $R[r\inv]_0$, and thus we obtain a Poisson structure on each member of an open cover of $X$. These structures are compatible on intersections, and  glue to give a Poisson structure on all of $X$. \end{proof}


\section{The wonderful compactification}\label{sec:classical}

In this section, we give an exposition of the construction of the Vinberg semigroup and the wonderful compactification. We present the main constructions in Sections \ref{subsec:vinberg} to \ref{subsec:poissonvinb}, and devote Section \ref{subsec:sl2} to a detailed discussion of the example of $\SL_2$. 

\subsection{The Vinberg semigroup}\label{subsec:vinberg} Fix a connected semisimple algebraic group $G$ over $\C$ with Lie algebra $\g$. We denote by $\Ug$ the universal enveloping algebra of $\g$, abbreviate the weight lattice $\Lambda_G$ by $\Lambda$, and adapt notation from Section \ref{subsec:notation}.

\begin{theorem}[Peter-Weyl Theorem] \label{thm:peterweyl} The map of matrix coefficients 
 $$\phi: \bigoplus_{\lambda \in \Lambda^+} V_\lambda^* \otimes V_\lambda \stackrel{\sim}{\longrightarrow} \O(G); \qquad f \otimes v \mapsto [g \mapsto f(g\cdot v)]$$ defines an isomorphism of $\Ug$-bimodules. \end{theorem}

While the Peter-Weyl theorem endows $\O(G)$ with a grading by $\Lambda$, the algebra structure on $\O(G)$ does not respect this grading. Instead, invoking the dominance order on $\Lambda$ from Definition \ref{def:gorder}, we obtain a $\Lambda$-filtered algebra:

\begin{defprop} \label{def:filtration} The subspaces $$ \O(G)_{\leq \lambda} =  \phi \left( \sum_{\mu\leq \lambda} V_\mu^* \otimes V_\mu \right),$$ for $\lambda \in \Lambda$, endow $\O(G)$ with the structure of a $\Lambda$-filtered algebra. That is, if $c_\lambda \in \O(G)_{\leq \lambda}$ and $c_\mu \in \O(G)_{\leq \mu}$, then the product $c_\lambda \cdot c_\mu$ lies in $\O(G)_{\leq \lambda + \mu}$. \end{defprop}

The fact that we obtain a filtered algebra is a consequence of Lemma \ref{lem:tensorproducts} and basic properties of matrix coefficients. We refer to this filtration on $\O(G)$ as the Peter-Weyl filtration, and note that it is not a filtration just by the cone $\Lambda^+$ of dominant weights. Indeed, there are non-dominant weights $\lambda$, such that the set  $\{ \mu \in \Lambda^+ \ | \ \mu \leq \lambda\}$ is nonempty, and hence the subspace $\O(G)_{\leq \lambda}$ is nonzero. This is true, for example, for any positive simple root. 

Let $\C[\Lambda]$ denote the group algebra of $\Lambda$ as an abelian group; it is generated by formal variables $z^\lambda$ with relations $z^\lambda z^\mu = z^{\lambda + \mu}$, for $\lambda, \mu \in \Lambda$.

\begin{definition} The Rees algebra for $\OG$ with the Peter-Weyl filtration is defined as  the following  $\Lambda$-graded subalgebra of $\O(G) \otimes \C[\Lambda]$: $$\ReesG =\bigoplus_{\lambda \in \Lambda} \O(G)_{\leq \lambda} z^\lambda.$$
\end{definition}

\begin{lemma}\label{lem:Rees} The $\Ug$-bimodule structure on $\OG$ extends to a $\Ug$-bimodule structure on $\Rees(G)$. The Hopf algebra structure on $\OG$ induces a bialgebra structure on $\Rees(G)$.  \end{lemma}

\begin{proof} To obtain the $\Ug$-bimodule structure, let $\Ug$ act trivially on $\C[\Lambda]$ and note that, by the Peter-Weyl theorem, each subset $\OG_{\leq \lambda}$ of $\OG$ is stable under the action of $\Ug \otimes \Ug$. The coproduct on $\O(G)$ restricts to a map $\O(G)_{\leq \lambda}$ $\rightarrow$ $\O(G)_{\leq \lambda} \otimes \O(G)_{\leq \lambda}.$ The coproduct and counit on $\Rees(G)$ are given in terms of the coproduct and counit on $\OG$, namely, for $f \in \OG_{\leq \lambda}$, we have $\Delta(fz^\lambda) = \Delta_{\OG}(f)\cdot z^\lambda \otimes z^\lambda$ and $\epsilon(fz^\lambda) =\epsilon_{\OG}(f)$.  \end{proof}

\begin{definition} The Vinberg semigroup $\Vinb_G$ for $G$ is defined as the spectrum of the Rees algebra for $\O(G)$ with the Peter-Weyl filtration: $$\Vinb_G = \Spec\left( \bigoplus_{\lambda \in \Lambda} \O(G)_{\leq \lambda} z^\lambda \right).$$ \end{definition}

Henceforth, we use $\ReesG$ and $\O(\Vinb_G)$ interchangeably. Lemma \ref{lem:Rees} implies that $\Vinb_G$ is a semigroup with an action of $G \times G$. One can show that the group of units of $\Vinb_G$ is the quotient of $G \times T$ by the antidiagonal multiplication action of the center $Z(G)$ of $G$. 

\begin{definition}\label{def:alphas} Let $\C[z^{\alpha_i}] =  \C[z^{\alpha_i}\ | \ i = 1, \dots, r]$ denote the polynomial subalgebra of $\C[\Lambda]$ generated by the elements $z^{\alpha_i}$ for $i \in \Delta$. Let $\A = \Spec \left(\C[z^{\alpha_i}]\right)$, so $\A$ is an $r$-dimensional affine space, and the choice of positive simple roots endows $\A$ with a coordinate system. 
\end{definition}

Observe that $z^{\alpha_i} \in \O(\Vinb_G)$ for any positive root $\alpha_i$, so there is an inclusion $\C[z^{\alpha_i}] \hookrightarrow \O(\Vinb_G)$. The induced surjective map \begin{equation} \pi: \Vinb_G \rightarrow \A \end{equation} is the abelianization map of \cite{Vinberg}. We record the following basic observations:
\begin{itemize}
\item Since $\Lambda = X^*(T)$ is the character lattice of the maximal torus $T$ of $G$, the algebra of functions $\O(T)$ on $T$ is the precisely the group algebra $\C[\Lambda]$.
 \item  The group algebra $\C[\Lambda_R]$ of the root lattice is a subalgebra of $\C[\Lambda]$, and can be identified with the algebra of functions $\O(T/Z)$ on the maximal torus $T/Z$ of the adjoint group $\Gad = G/Z$.
  \item The polynomial algebra $\C[z^\alpha_i \ | \ i = 1, \dots, r]$ of Definition \ref{def:alphas} is a subalgebra of $\C[\Lambda_R] = \O(T/Z)$. Therefore, the affine space $\A$ is a toric variety of $T/Z$. In particular, $\A$ has an action of $T$. 
\item The map $\pi : \Vinb_G \rightarrow \A$ is $T$-equivariant. 
\end{itemize}

\begin{example} The Vinberg semigroup for $\SL_2$ is the semigroup of two by two matrices, and the map $\pi$ is the determinant map, as explained in Section \ref{subsec:sl2} below. \end{example}

\begin{rmk} See \cite[Example 3.2.4]{Brion} for the relation between the definition of the Vinberg semigroup presented in this section and Vinberg's original definition; the latter appears in \cite{Vinberg}. See \cite[Section D.2.3]{DrinfeldGaitsgory} for a Tannakian approach to defining the Vinberg semigroup through its category of representations. \end{rmk}

\subsection{The wonderful compactification}\label{subsec:wonderful}

In this section, we describe two realizations of the wonderful compactification. The first is based on work of De Concini and Springer \cite{DCS}, as well as an exposition due to Evens and Jones \cite{EvensJones}, and proceeds as follows. Fix an irreducible representation $V = V_\lambda$ of $G$ of regular highest weight $\lambda \in \Lambda^+$. Consider the commutative diagram:\[ \xymatrix{ G \ar[r] \ar[d] &  \GL(V) \ar@{^{(}->}[r] & \End(V) \setminus \{0\} \ar[d] \\
\Gad \ar@{-->}[rr]^\psi && \P(\End(V)) },\] 
where the left horizontal arrow is the action map, the right vertical arrow is the quotient by the free action of $\C^\times$, and the remaining solid arrows represent obvious maps. The existence of the map $\psi$ relies on the fact that the center $Z(G)$ acts on $V$ by scalars. A consideration of the weightspace decomposition of $V$ is used to prove the following lemma:

\begin{lemma} The map $\psi$ is injective and equivariant for the action of $G \times G$. \end{lemma}

\begin{definition} The wonderful compactification of $\Gad$ is defined as the closure $\overline{\psi(\Gad)}$ of the image of $\psi$ in $\P(\End(V))$. \end{definition}

\begin{prop}\cite[Propositions 2.14, 3.1]{EvensJones}\label{prop:smoothprojindependent} The wonderful compactification $\Gbarad$ is a smooth projective $G \times G$ variety. Up to isomorphism, it does not depend on the choice of regular dominant weight.  \end{prop}

\begin{example} The wonderful compactification of $\PSL_2$ is $\P^3$ (see Section \ref{subsec:sl2} below for more details). However, for $n \geq 3$, the wonderful compactification of $\PSL_n$ is not $\P^{n^2-1}$.  \end{example}

To explain a second realization of the wonderful compactification of $\Gad$, we first observe that, since $\Lambda = X^*(T)$ is the character lattice of the maximal torus $T$ of $G$, the $\Lambda$-grading on the Rees algebra $\O(\Vinb_G)$ endows $\Vinb_G$ with a $T$-action which commutes with the $G \times G$-action. We have the following result:

\begin{theorem}{\cite[Theorem 5.3]{MartensThaddeus}} \label{thm:MT} Fix a regular dominant weight $\lambda$ of $G$, thought of as a character of $T$. The GIT quotient of $\Vinb_G$ by $T$ along $\lambda$ recovers the wonderful compactification: $$\Gbarad = \Vinb_G \text{ $\!$/$\! \!$/$\!$}_\lambda  T.$$ \end{theorem}

Proposition \ref{prop:qcohproj} implies the following corollary of the preceding theorem:

\begin{cor}\label{cor:qcohgadbar} There is an equivalence of categories: $$\QCoh\left(\Gbarad\right) = \uProj\left( \O(\Vinb_G) \right).$$ \end{cor}

\begin{rmk} The total coordinate ring, or Cox ring, of $\Gbarad$ is precisely $\O(\VinbG)$ (see \cite{Brion}). \end{rmk}

\subsection{Orbits in the wonderful compactification} \label{subsec:orbits}

In this section, we describe the $G \times G$ orbits on the wonderful compactification. Fix a subset $I \subseteq \Delta$ and consider the corresponding parabolic subgroup $P_I$, its opposite $P_I^-$, and its Levi $L_I$. There are projection maps $\text{\rm pr}^L : P_I \rightarrow L_I$ and $\text{\rm pr}^{L^-} : P_I^- \rightarrow L_I$, and each of these compose to a map valued in $L_I^{\text{ad}} = L_I/ Z(L_I)$. 

\begin{prop}\cite{EvensJones}\label{prop:ejorbits} We have:
\begin{enumerate}
\item The $G \times G$ orbits on $\Gbarad$ are in bijection with subsets of $\Delta$. Write $\Orb_I$ for the orbit corresponding to $I \subseteq \Delta$. For subsets $I$ and $J$ of $\Delta$, the containment $\Orb_I \subseteq \overline{\Orb_J}$ holds if and only if $I \subseteq J$. 

\item There is a point in $\Orb_I$ whose stabilizer is the subgroup
$$H_I = P_I \times_{L_I^{\text{ad}}} P_I^- = \{ (p,p^-) \in P_I \times P_I^- \ | \ \text{\rm pr}^L(p) \text{\rm pr}^{L^-}(p^-)\inv \in Z(L_I)\}.$$

\item Let $\overline{\Lad}$ denote the wonderful compactification of the adjoint group $\Lad$ of $L_I$, and let  $\overline{\Orb_I}$ denote the closure of $\Orb_I$ in $\Gbarad$. There are fibrations:

\[ \xymatrix{ \Lad \ar[r]  & \Orb_I \ar[d] \\ & G/P_I \times P_I^- \backslash G  }  \qquad \xymatrix{ \overline{\Lad} \ar[r]  & \overline{\Orb_I} \ar[d] \\ & G/P_I\times P_I^- \backslash G  }. \]
\end{enumerate} \end{prop}

\begin{example} In the extreme cases,  $\Orb_\Delta = \Gad$ is the unique open orbit, and $\Orb_\emptyset$ $=$ $G/B \times B^-\backslash G$ is the unique closed orbit. When $G = \SL_2$, there are only two $G \times G$ orbits on $\Gbarad = \P^3$, and they are the extreme ones. In the $\SL_3$ case, there are four orbits. The two non-extreme orbits each form $\PSL_2$-bundles over $\P^2 \times \P^2$, and the closure of each forms a $\P^3$-bundle over  $\P^2 \times \P^2$. \end{example}

Let $e_I$ be the point in $\A$ whose $i$th coordinate is zero if $i \notin I$ and 1 otherwise. 
Recall the $T$-equivariant  map $\pi : \Vinb_G \rightarrow \mathbb{A}$ defined in Section \ref{subsec:vinberg}. The preimage $\pi\inv(T \cdot e_I)$ of the $T$-orbit of $e_I$ is an affine subvariety of $\Vinb_G$ with an action of $T$.

\begin{prop}\label{prop:orbitsGIT} Fix a regular dominant weight $\lambda \in \Lambda^+$. For each subset $I \subseteq \Delta$, the orbit $\Orb_I$ is the GIT quotient of $\pi\inv(T\cdot e_I)$ by $T$ at $\lambda$: 
$$\Orb_I =   \pi\inv(T\cdot e_I) \text{ $\!$/$\! \!$/$\!$}_{\lambda} T.$$ 
\end{prop}

\begin{proof} The space $\mathbb{A}$ is the union of the $T$-orbits of $e_I$, as $I$ ranges over the subsets of $\Delta$, and these orbits are pairwise disjoint. Therefore, the Vinberg semigroup $\Vinb_G$ is the union of the subvarieties $\pi\inv(T\cdot e_I)$. Since the action of $T$ on $\Vinb_G$ commutes with that of $G \times G$, each of these subvarieties carries an action of $G\times G$. Using Theorem \ref{thm:MT}, we see that the wonderful compactification $\Gadbar$ is the union of the disjoint $G\times G$-equivariant subvarieties $\pi\inv(T\cdot e_I) \text{ $\!$/$\! \!$/$\!$}_{\lambda} T$,  as $I$ ranges over subsets of $\Delta$. The result follows from the fact that the $G \times G$ orbits on $\Gbarad$ are in bijection with subsets of $\Delta$ (Proposition \ref{prop:ejorbits}). \end{proof}

The algebra of functions $\O(\pi\inv(T \cdot e_I))$ on $\pi\inv(T \cdot e_I)$ is a $\Lambda$-graded algebra. Proposition \ref{prop:qcohproj} implies the following corollary of the preceding proposition:

\begin{cor} The category of quasicoherent sheaves on $\Orb_I$ is equivalent to the $\uProj$ category for the $\Lambda$-graded algebra $\O(\pi\inv(T \cdot e_I))$, that is: $$\QCoh(\Orb_I) = \uProj\left(  \O(\pi\inv(T\cdot e_I)) \right).$$  \end{cor}

\subsection{Poisson structures}\label{subsec:poissonvinb}

In this section, we describe how the standard Poisson-Lie structure on the group $G$ leads to a Poisson structure on the Vinberg semigroup $\VinbG$, on the wonderful compactification $\Gadbar$, and on each orbit $G \times G$ orbit $\Orb_I$ in $\Gadbar$. We adopt the notation of Section \ref{subsec:notation}. 

\begin{lemma}\label{lem:poissonfiltration} The standard Poisson-Lie bracket on $\O(G)$ preserves the Peter-Weyl filtration on $\O(G)$; that is, it restricts to a map $$\{, \} : \O(G)_{\leq \lambda} \otimes \O(G)_{\leq \mu} \rightarrow \O(G)_{\lambda + \mu}.$$ 
\end{lemma}

\begin{proof} Let $X \in \g$ and $c \in \O(G)$. Write $X^L c$ and $X^R c$ for the image of $c$ under the action of $X$ by left- and right-invariant vector fields. If $c = c_{f,v}$ is a matrix coefficient for a representation $V$ of $\g$, then $X^L c_{f, v} = c_{Xf, v}$ and $X^R c_{f,v} = c_{f, Xv}$. Thus, the actions $X^R$ and $X^L$ preserve the space of matrix coefficients for a given representation. The result now follows from Proposition \ref{prop:poissonlie}. \end{proof}

\begin{definition}\label{def:vinbergpoisson} Define a map $\{, \} : \O(\VinbG) \otimes \O(\VinbG) \rightarrow \O(\VinbG)$ by 
$$\{a_\lambda z^\lambda, a_\mu, z^\mu \} = \{a_\lambda, a_\mu\} z^{\lambda + \mu},$$  
where $a_\lambda \in \O(G)_{\leq \lambda}$, $a_\mu \in \O(G)_{\leq \mu}$, and the right-hand side invokes the standard Poisson-Lie bracket on $\O(G)$.
\end{definition}

\begin{prop}\label{prop:poisson} We have:
\begin{enumerate}
 \item The map defined above is a Poisson bracket on $\O(\VinbG)$, compatible with the $\Lambda$-grading.
 \item There is an induced Poisson structure on the wonderful compactification $\Gadbar$. 
 \item For $I \subseteq \Delta$, the fiber $\pi\inv(e_I)$ is a Poisson subvariety of $\VinbG$, and so is its $T$-orbit. Consequently, there is an induced Poisson structure on the orbit $\Orb_I$.
\end{enumerate} \end{prop}

\begin{proof} The first statement follows from Lemma \ref{lem:poissonfiltration}. The second statement follows from Lemma \ref{lem:poissonGIT} and Theorem \ref{thm:MT}. Endow $\O(\A)$ with the zero Poisson bracket, so that any point of $\A$ is a Poisson subvariety. The map $\O(\A) \rightarrow \O(\VinbG)$ is a map of Poisson algebras. This implies that the fiber $\pi\inv(e_I)$ and the preimage $\pi\inv(T \cdot e_I)$ of the $T$-orbit of $e_I$ are each affine Poisson subvarieties of $\VinbG$. The Poisson bracket on the algebra of functions $\O(\pi\inv(T \cdot e_I))$ is compatible with the $\Lambda$-grading. The third statement now follows from Lemma \ref{lem:poissonGIT} and Proposition \ref{prop:orbitsGIT}. \end{proof}

\begin{rmk} The methods of the proof can be used to show that, at least when $G$ is simply-connected, any Lie bialgebra structure on $\g$ leads to a Poisson structure on the Vinberg semigroup, on the wonderful compactification, and on any orbit. Poisson structures on the wonderful compactification have been studied extensively by Lu and Yakimov \cite{LY1, LY2}.
\end{rmk}

\subsection{The case of $\SL_2$} \label{subsec:sl2}
In this section, we describe in detail the constructions of the previous sections for the case of $\SL_2$. 

Identify the maximal torus of $\SL_2$ with $\C^\times$, the weight lattice $\Lambda_{\SL_2}$  with the integers $\Z$, the cone of dominant weights $\Lambda_{\SL_2}^+$ with the cone of nonnegative integers, and the root lattice $\Lambda_R$ with the sublattice $2\Z$ generated by the positive simple root $\alpha_1 = 2$. A dominant weight (i.e.\ nonnegative integer) $n$ is regular if and only if $n$ is nonzero. The partial order on $\Z$ from Definition \ref{def:gorder} becomes $$m \leq n \ \text{if and only if $n-m$ is a nonnegative multiple of 2}.$$  For any nonnegative integer $n$, write $V_n = \Sym^n (\C^2)$ for the irreducible representation of $\SL_2$ of highest weight $n$. As a special case of Lemma \ref{lem:tensorproducts}, we have that, if $V_k$ appears as an irreducible representation of the tensor product $V_n \otimes V_m$, then $k \leq n+ m$. 

The algebra of functions $\O(\SL_2)$ is given by $\O(\SL_2) = \C[a,b,c,d]/( ad-bc =1 ).$  This is a Hopf algebra with $$ \Delta(a) = a \otimes a + b \otimes c, \quad \Delta(b) = a \otimes b + b \otimes d, \quad \Delta(c) = a \otimes c + c \otimes d, \quad \Delta(c) =  b\otimes c + d \otimes d,$$ $$ \epsilon(a) =\epsilon(d) = 1, \quad \epsilon(b) = \epsilon(c) =0, \quad S(a) = d, \quad S(b) = -b, \quad S(c) = -c, \quad S(d) = a.$$ We will use the same notation for elements of $\C[a,b,c,d]$ and their images in $\O(\SL_2)$. 

\begin{theorem}[Peter-Weyl Theorem] The map of  matrix coefficients gives an $\SL_2 \times \SL_2$-equivariant isomorphism: $\phi: \bigoplus_{n} V_n \otimes V_n^* \stackrel{\sim}{\longrightarrow} \O(SL_2).$ \end{theorem}

The Peter-Weyl filtration on $\O(\SL_2)$ is given by $\O(\SL_2)_{\leq n} =  \phi\left( \bigoplus_{m \leq n} V_m^* \otimes V_m \right)$ for $n \in \Z$. The proof of the following lemma is straightforward.

\newcommand{\pr}{\text{\rm pr}}

\begin{lemma}\label{lem:leqn} For $n \in \Z$, consider the subspace of $\C[a,b,c,d]$ spanned by monomials $a^{k_1} b^{k_2} c^{k_3} d^{k_4}$ with $ k_1+k_2 + k_3 + k_4  \leq n$ (in the usual partial order on $\Z$) and $k_1+k_2 + k_3 + k_4  \equiv n \mod 2.$ The image of this subspace under the quotient map $\C[a,b,c,d] \rightarrow \O(\SL_2)$ is precisely  $\O(\SL_2)_{\leq n}$. \end{lemma}

For example:
\begin{itemize}
\item $\O(\SL_2)_{\leq 0} = \phi(V_0 \otimes V_0^*) =  \text{Span}\{1\}$
\item $\O(\SL_2)_{\leq 1} = \phi(V_1 \otimes V_1^*) = \text{Span}\{a, b,c,d\}$
\item $\O(\SL_2)_{\leq 2} = \phi(V_0 \otimes V_0^* \oplus V_2 \otimes V_2^*) = \text{Span}\{1, a^2, b^2, c^2, d^2, ab, ac, \dots\}$
\item $\O(\SL_2)_{\leq 3} = \phi(V_1 \otimes V_1^* \oplus V_3 \otimes V_3^*) = \text{Span}\{ \text{\rm degree 3 and degree 1 monomials}\}$
\end{itemize}

Recall from Section \ref{subsec:projbasic} our convention that the notation `$\Proj$' in roman font indicates a projective scheme (i.e.\ a space), while `$\uProj$' in sans serif font indicates a category.

\begin{prop}\label{prop:assgrad} The associated graded algebra $\text{\rm gr}(\O(\SL_2))$ is the homogeneous coordinate ring of $\P^1 \times \P^1$. Consequently, we have an isomorphism of varieties: $$\Proj( \text{\rm gr}(\O(\SL_2))) \simeq  \P^1 \times \P^1.$$ \end{prop}

\begin{proof} First, $\text{\rm gr}(\O(\SL_2)) \simeq \C[a,b,c,d]/(ad-bc)$. The latter is isomorphic as a graded algebra to $\bigoplus_{n}$ $\Sym^n(\C^2) \otimes \Sym^n(\C^2)$ via the map $a \mapsto x \otimes w$, $b \mapsto x \otimes z$, $c \mapsto y \otimes w$, and $d \mapsto y \otimes z,$ where $x,y$ are the coordinate functions in the first copy of $\Sym^\bullet(\C^2)$ and $w,z$ are the coordinates in the second copy. Now, $\Sym^n(\C^2) \otimes \Sym^n(\C^2)$ is the space of global sections of the line bundle $\O(n,n) = \O(1,1)^{\otimes n}$ on $\P^1\times \P^1$. The line bundle $\O(1,1)$ is ample since it is the pullback of $\O_{\P^3} (1)$ under the Segre embedding $\P^1 \times \P^1 \rightarrow \P^3$. Therefore, there is an isomorphism $\P^1 \times \P^1 \stackrel{\sim}{\longrightarrow} \Proj\left(  \bigoplus_n \Gamma \left( \P^1 \times \P^1, \O(1,1)^{\otimes n}\right)\right).$  \end{proof}

 \newcommand{\ReesSL}{\text{\rm Rees}(\SL_2)}

\begin{prop}\label{prop:mainresult} We have the following:
\begin{enumerate}
\item The Vinberg semigroup $\Vinb_{\SL_2}$ is isomorphic to the semigroup  $\Mat_2$ of two by two matrices, and the action of $\C^\times$ on $\Vinb_{\SL_2}$ coincides with the scaling action on $\Mat_2$.

\item The map $\pi : \Vinb_{\SL_2} \rightarrow \mathbb{A}$ coincides with the determinant map $\det : \Mat_2 \rightarrow \C$. 

\item The wonderful compactification of $\PSL_2$ is $\P^3 = \Proj(\O(\Vinb_{\SL_2}))$.

\item \label{prop:classicalorbits} The stratification of $\barPSL = \P^3$ into $\SL_2 \times \SL_2$ orbits is given by:  $$ \overline{\PSL_2} \ = \  \PSL_2  \ \coprod \ ( \P^1 \times \P^1).$$ 
\end{enumerate}
\end{prop}

\begin{proof} The algebra of functions on the Vinberg semigroup $\Vinb_{\SL_2}$ is given by the Rees algebra  $\O(\Vinb_{\SL_2})$ $=$ $\bigoplus_{n \geq 0} $ $\O(\SL_2)_{\leq n} z^n.$  Lemma \ref{lem:leqn} and the relation $z^2 = (az)(dz) - (bz)(cz)$ together imply that the subspace $\O(\SL_2)_{\leq n} z^n$ coincides with the span  of monomials $$(az)^{k_1^\prime} (bz)^{k_2^\prime} (cz)^{k_3^\prime} (dz)^{k_4^\prime} $$ with  $k_1^\prime+k_2^\prime + k_3^\prime + k_4^\prime  \leq n$ (under the usual partial order on $\Z$). Hence, $\O(\Vinb_{\SL_2})$ is a commutative algebra on the four generators $az$, $bz$, $cz$, $dz$ and no relations. It is straightforward to verify that the coproduct of $\O(\Vinb_{\SL_2})$ coincides with that of $\O(\Mat_2)$. Since $\Vinb_{\SL_2} = \Spec(\O(\Vinb_{\SL_2}))$, this proves the first statement. 

The second statement follows from the relation $z^2 = (az)(dz) - (bz)(cz)$. Fix the regular dominant weight $n =1$, regarded as a character of the maximal torus $T = \C^\times$ of $\SL_2$. The third statement follows from the computation:
$$\overline{\PSL_2}  = \VinbSL \text{ $\!$/$\! \!$/$\!$}_1  \C^\times =   \Mat_2 \text{ $\!$/$\! \!$/$\!$}_1  \C^\times =( \Mat_2 \setminus \{0\})/\C^\times = \P^3.$$ 

To prove the last statement, first note that each $\SL_2 \times \SL_2$ orbit on $\Mat_2 \setminus \{0\}$ has the form $\det\inv(d)$ for some $d \in \C$, where $\det :  \Mat_2 \setminus \{0\} \rightarrow \C$ is the determinant map. If $d \neq 0$, then one uses the $\C^\times$-action on  $\Mat_2 \setminus \{0\}$ to identify $\det\inv(d)$ with $\det\inv(1) =\SL_2$. It follows that the disjoint union $\coprod_{d \in \C^\times} \det\inv(d)$ descends to a single $\SL_2 \times \SL_2$-orbit in $\barPSL$, and it is isomorphic to $\PSL_2$. On the other hand, if $d=0$, then we have:
$$ \det {}^{-1}(0) = \{ \text{\rm matrices of rank 1}\} \simeq \frac{\C^2\setminus \{0\} \times \C^2\setminus \{0\}}{\C^\times}$$
The second identification is given by sending pair of nonzero vectors $(x_1, x_2)$  and $(y_1, y_2)$ to the matrix whose $(i,j)$ entry is $x_i y_j$ for $i,j \in \{1,2\}.$ The $\C^\times$ action on $\Mat_2 \setminus \{0\}$ preserves this space, and the quotient is $\P^1 \times \P^1$. \end{proof}

We give another perspective on the $\SL_2 \times \SL_2$ orbits on  the wonderful compactification $\overline{\PSL_2} = \P^3 = \Proj( \O(\Vinb_{\SL_2}))$. First note that $\O(\Vinb_{\SL_2})$ contains $z^2$, but does not contain $z$. We have
\begin{equation}\label{eq:SL2orbits} \overline{\PSL_2}  =  \Spec  \left( \O(\Vinb_{\SL_2})[(z^2)\inv]^{\C^\times}\right)   \coprod   \Proj  \left(\O(\Vinb_{\SL_2})/  (z^2) \right).\end{equation}
It is straightforward to verify that $\O(\Vinb_{\SL_2})[(z^2)\inv] =  (\O(\SL_2)[z^{\pm 1}])^{\Z/2\Z} = \O(\GL_2)$.  The action of $\C^\times$ on $\O(\Vinb_{\SL_2})[(z^2)\inv]$ corresponds to the (free) action on $\GL_2$ by its center. Thus, $$\Spec  \left(\O(\Vinb_{\SL_2})[(z^2)\inv]^{\C^\times}\right) = \GL_2/\C^\times = \PSL_2.$$ 
A standard argument shows that $\O(\Vinb_{\SL_2})/ (z^2) = \gr(\O(\SL_2))$. By Proposition \ref{prop:assgrad}, we have that $ \Proj  \left(\O(\Vinb_{\SL_2})/ (z^2)\right)  = \P^1 \times \P^1.$ We see that the decomposition (\ref{eq:SL2orbits}) becomes precisely the decomposition of Proposition \ref{prop:mainresult}. 

\begin{rmk} The orbit $\P^1 \times \P^1$ includes in $\overline{\PSL_2} = \P^3$ as the Segre embedding. \end{rmk}

We conclude this section with a brief discussion of Poisson structures. The standard Poisson-Lie bracket on $\O(\SL_2)$ is given by $$ \{a,b\} = ab,  \quad \{a,c\} = ac,  \quad \{b,c\} = 0, \quad \{b,d\} = bd, \quad \{c,d\} = cd, \quad \{a,d\} = 2bc.$$ 
The corresponding Poisson bracket on $\O(\Vinb_{\SL_2})$ is given by
$$ \{az,bz\} = abz^2, \quad \{az,cz\} = acz^2, \quad \{bz,cz\} = 0, \quad \{bz,dz\} = bdz^2,$$ $$\{cz,dz\} = cdz^2,  \quad \{az,dz\} = 2bcz^2.$$


\section{The wonderful compactification for quantum groups}\label{sec:quantumwc}

This section is the heart of the paper, in which we define the quantum coordinate algebra $\O_q(\Vinb_G)$ of Vinberg semigroup and the category of sheaves on the quantum wonderful compactification. We adapt notation from previous sections, especially Section \ref{sec:preliminaries}.

\subsection{Main definitions}\label{subsec:qdefs}

Let $G$ be a connected semisimple algebraic group over $\C$ with Lie algebra $\g$. 

\begin{prop}\label{prop:oqgfiltration}  There is an isomorphism of $\uqg$-bimodules $$\phi :  \bigoplus_{\lambda \in \Lambda^+} \mathcal V_\lambda \otimes \mathcal V_\lambda^* \stackrel{\sim}{\longrightarrow} \OqG,$$ where the sum ranges over irreducible objects in $\mathcal C_q(\g)$ with highest weights in $\Lambda_G^+ = \Lambda^+$. The subspaces $$ \OqG_{\leq \lambda} =  \phi \left( \sum_{\mu\leq \lambda} \mathcal V_\mu^* \otimes \mathcal V_\mu \right),$$ for $\lambda \in \Lambda$, endow $\OqG$ with the structure of a $\Lambda$-filtered algebra. \end{prop}

\begin{proof} The first statement is a consequence of the definition of the quantum coordinate algebra $\O_q(G)$ (Definition \ref{def:oqg}). General properties of matrix coefficients, together with Lemma \ref{lem:qtensorproducts} imply that the subspaces $\OqG_{\leq \lambda}$ form a filtration. \end{proof}

\begin{definition}\label{def:maindefs} We make the following definitions:
\begin{enumerate}
 \item The quantized coordinate algebra $\OqVinb$ of the Vinberg semigroup for $G$ is defined as the Rees algebra for $\OqG$, that is, as the following $\Lambda$-graded subalgebra of $\OqG \otimes \C[\Lambda]$: $$\OqVinb = \bigoplus_{\lambda \in \Lambda} \OqG_{\leq \lambda} z^\lambda.$$
 \item  The category of quasicoherent sheaves on the quantum wonderful compactification of $\Gad$ is given by $$\QCoh_q\left(\Gbarad\right) = \uProj\left( \OqVinb\right).$$   
\end{enumerate}
\end{definition}

\begin{prop}\label{prop:qvinberg} The algebra $\OqVinb$ has a natural bialgebra structure, forms a flat deformation of the coordinate algebra $\O(\Vinb_G)$ of the Vinberg semigroup, and quantizes the Poisson bracket of Definition \ref{def:vinbergpoisson}. When $q=1$, we recover from $\QCoh_q\left( \Gbarad \right)$ the category of quasicoherent sheaves on the wonderful compactification $\Gbarad$.
\end{prop}

\begin{proof} The statements about $\OqVinb$ follow from Proposition \ref{prop:oqgfiltration} and the fact that $\OqG$ is a flat deformation of $\OG$, quantizing the standard Poisson-Lie bracket. The final assertion follows from Corollary \ref{cor:qcohgadbar} (but ultimately from \cite{MartensThaddeus, Serre}). \end{proof}

Recall from Definition \ref{def:alphas} that we denote by $\A$ the spectrum of the polynomial subalgebra $\C[z^{\alpha_i}]$ of $\C[z^\lambda]$ generated by the elements $z^{\alpha_i}$ for $i \in \Delta$. The fact that $\alpha_i \geq 0$ for any positive simple root $\alpha_i$ implies the following:

\begin{lemma}\label{lem:qVinbbasics} For any $i\in \Delta$, the element $z^{\alpha_i}$ belongs to $\OqVinb$, and is central. Therefore, $\OqVinb$ defines a sheaf of algebras on $\mathbb{A}$. \end{lemma}
 
The choice of positive simple roots endows $\mathbb{A}$ with a coordinate system. Over a point where all coordinates are nonzero, the fiber of the sheaf defined by $\OqVinb$ is isomorphic to $\OqG$. The fiber at a general point is a certain `partial' associated graded algebra for $\OqG$; we describe these algebras in the next two sections.  

\subsection{Filtrations on the quantum coordinate algebra}\label{subsec:qfiltrations}

In this section, we consider certain filtrations on the quantized coordinate algebra $\OqG$ of $G$, and describe the associated graded algebras. These associated graded algebras define the quantum orbits on the wonderful compactification.

Given a subset $I \subseteq \Delta$ of positive simple roots let $\Lambda_I = \Z\{\alpha_i \ | \ i \in I\} \subseteq \Lambda_R$ denote the sublattice of the root lattice spanned by the roots in $I$. Write $[\lambda]$ or $[\lambda]_I$ for the image of $\lambda \in \Lambda$ in $\Lambda/\Lambda_I$.  Define a partial order on $\Lambda/\Lambda_I$ by 
\[ \text{$[\mu]_I \leq [\lambda]_I$ whenever $\lambda - \mu = \sum_{i=1}^r n_i \alpha_i$ with $n_i \in \Z$ if $i \in I$ and $n_i$ is nonnegative if $i \notin I$.} \]

\newcommand{\qfilterlambda}{\OqG_{\leq [\lambda]_I}}
\newcommand{\qfiltermu}{\OqG_{\leq [\mu]_I}}
\newcommand{\qfilterlm}{\OqG_{\leq [\lambda+\mu]_I}}

\begin{definition} \label{def:Ifiltration} For $\lambda \in \Lambda$, define the following subspace of $\OqG$: 
$$\qfilterlambda =  \phi \left( \sum_{[\mu]_I \leq [\lambda]_I } \mathcal V_\mu^* \otimes \mathcal V_\mu \right).$$
\end{definition}

 If $I= \emptyset$, we recover the partial order on $\Lambda$ from Definition \ref{def:gorder}, and $\OqG_{\leq [\lambda]_\emptyset}$ coincides with the subspace $\OqG_{\leq \lambda}$ from Proposition \ref{prop:oqgfiltration}.

\begin{prop}\label{prop:Ioqgfiltration}
The subspaces $\qfilterlambda$  define a filtration on $\OqG$ by $\Lambda/\Lambda_I$. The associated graded algebra has $[\lambda]_I$-graded piece equal to 
$$\phi \left( \bigoplus_{\nu \in \Lambda_I} \mathcal V_{\lambda + \nu}^* \otimes \mathcal V_{\lambda + \nu} \right).$$ 
The coproduct $\Delta$ restricts to a map $\Delta : \qfilterlambda \rightarrow \qfilterlambda \otimes \qfilterlambda.$  \end{prop}

\begin{proof} The fact that the subspaces $\qfilterlambda$ form a filtration follows from general properties of matrix coefficients, together with Lemma \ref{lem:tensorproducts}. For $\lambda \in \Lambda$, the $[\lambda]_I$-th graded piece of the associated graded algebra is given by
$$\frac{\qfilterlambda}{\sum_{[\mu] \leq [\lambda], [\mu] \neq [\lambda]} \qfiltermu}=  \phi \left( \frac{ \bigoplus_{[\mu] \leq [\lambda] } \mathcal V_\mu^* \otimes \mathcal V_\mu  }{  \bigoplus_{[\mu] \leq [\lambda], [\mu] \neq [\lambda]}  \mathcal V_{\mu}^* \otimes \mathcal V_{\mu}   } \right) = \phi \left( \bigoplus_{[\mu] = [\lambda]} \mathcal V_\mu^* \otimes \mathcal V_\mu \right).$$ The set of $\mu \in \Lambda$ with $[\mu] = [\lambda]$ is precisely $\{ \lambda + \nu \ | \ \nu \in \Lambda_I\}$. The final claim is a consequence of Lemma \ref{lem:matrixcoeff}.\ref{lem:matrixcoeffbasics}. \end{proof}

\begin{definition}\label{def:qorbits} For $I \subseteq \Delta$, let $\gr_I(\O_q(G) )$ denote the associated graded algebra of $\O_q(G)$ with the filtration of Definition \ref{def:Ifiltration}. \end{definition}

\begin{example} If $I = \emptyset$, we obtain the full associated graded algebra of $\OqG$. That is, $\gr_{\emptyset} (\OqG) = \bigoplus_{\lambda \in \Lambda^+} \phi \left( \mathcal V_\lambda^* \otimes \mathcal V_\lambda\right).$ At the other extreme, if $I = \Delta$, then $\gr_\Delta(\OqG)$  is isomorphic as an algebra to $\OqG$, and its grading coincides with the grading of $\OqG$ by the finite group  $\Lambda/\Lambda_R$ (which can be naturally identified with the dual of the center $Z(G)$ of $G$):
$$\bigoplus_{[\lambda] \in \Lambda/\Lambda_R}  \left( \sum_{\nu \in \Lambda_R} \OqG_{\leq \lambda + \nu} \right).$$ \end{example}

Observe that $\OqG$ and $\gr_I(\OqG)$ are isomorphic as $\uqg$-bimodules. The algebra $\gr_I(\OqG)$ is a `partial' associated graded algebra, and its multiplication map can be described more explicitly as the composition of the ordinary multiplication map $$\OqG_{\leq \lambda} \otimes \OqG_{\leq \mu} \rightarrow \OqG_{\leq \lambda + \mu} = \phi \left( \bigoplus_{\nu \leq \lambda + \mu} \mathcal V_\nu^* \otimes \mathcal V_\nu\right)$$ with the projection onto the partial sum of the images of those $\mathcal V_\nu \otimes \cV_\nu^*$ such that $\lambda + \mu - \nu$ lies in $\Lambda_I$. 

Recall that $e_I$ denotes the point of $\mathbb{A}$ whose $i$th coordinate is zero if $i \notin I$ and 1 otherwise (Section \ref{subsec:orbits}), and that there is a map $\pi : \Vinb_G \rightarrow \mathbb{A}$ whose preimage over the $T$-orbit of $e_I$ is recovers the orbit $\Orb_I$ as a GIT quotient (Proposition \ref{prop:orbitsGIT}). 

\begin{definition} Let $(\Orb_I)_q = \O_q(\Vinb) \ot_{\O(\mathbb{A})} \O(T \cdot e_I)$ be the $\Lambda$-graded algebra of sections over the orbit $T \cdot e_I \subseteq \mathbb{A}$ of the sheaf  defined by $\OqVinb$ (recalling Lemma \ref{lem:qVinbbasics}). \end{definition}

\begin{prop} The  $\Lambda$-graded algebra $(\Orb_I)_q$ has the following properties:
\begin{itemize}
\item It is isomorphic to the $\Lambda$-graded algebra $\gr_I(\OqG) \ot \C[\Lambda_I]$. 
\item It forms a flat deformation of the algebra $\O(\pi\inv(T \cdot e_I))$.
\item It quantizes the Poisson bracket on $\O(\pi\inv(T \cdot e_I))$ considered in Proposition \ref{prop:poisson}.
\end{itemize}
\end{prop}

\begin{proof} The first statement follows from standard Rees algebra formalism, noting that the algebra $\gr_I(\OqG) \ot \C[\Lambda_I]$ carries a natural grading by $(\Lambda/\Lambda_I) \times \Lambda_I = \Lambda$. The second and third statements follow from the corresponding statements about $\OqVinb$ and the Vinberg semigroup, see Proposition \ref{prop:qvinberg}.\end{proof}

\begin{definition} For $I \subseteq \Delta$, the category of quasicoherent sheaves on the quantum $\Orb_I$ is given by $$\QCoh_q\left(\Orb_I \right) = \uProj\left( (\Orb_I)_q \right),$$ or, equivalently, by $\uProj\left(\gr_I(\OqG) \ot \C[\Lambda_I] \right).$ \end{definition}

We conclude this section by explaining the relation between the full associated graded algebra $\gr_\emptyset (\OqG) = \gr(\OqG)$ and the quantum flag variety of Backelin and Kremnitzer \cite{BaKr}. The quantum coordinate algebra $\O_q(G)$ for $G$ is a comodule for the quantum coordinate algebra $\O_q(B)$ of a Borel subgroup $B$ of $G$, and the algebra structure on $\O_q(G)$ is compatible with the comodule structure. In other words, $\O_q(G)$ is an algebra object in the category of $\O_q(B)$-comodules. Hence we can consider the category $\cM_{B_q}(G_q)$ of $\O_q(G)$-modules in the category of $\O_q(B)$-comodules. This category is a $q$-deformation of the category of quasicoherent sheaves on the flag variety $G/B$, and is known as the quantum flag variety. A doubled version of this category is given by the category $\cM_{B_q \times B_q^-}(G_q \times G_q)$ of $\O_q(G \times G)$-modules within the category of $\O_q(B \times B^-)$-comodules. The following theorem is a consequence of work of Backelin and Kremnitzer \cite[Corollary 3.7]{BaKr}. 

\begin{theorem} There is an equivalence of categories $$\cM_{B_q \times B_q^-}(G_q \times G_q) \simeq \uProj(\gr(\O_q(G))).$$ \end{theorem}

\subsection{The quantum orbits}\label{subsec:qorbits}

In this section, we provide a different description of the partial associated graded algebras $\gr_I( \OqG)$, and hence of the category of sheaves on each quantum orbit. 

Fix a subset $I \subseteq \Delta$. Throughout, we abbreviate $\frakl_I$ by $\frakl$ and $\fraku_I$ by $\fraku$. The map
$$ U_q(\fraku \times \mathfrak{l} \times \fraku^-) = U_q(\fraku)  \otimes U_q(\mathfrak{l})  \otimes U_q(\fraku^-) \rightarrow U_q(\g \times \g)$$
$$ x \otimes y \otimes z \rightarrow xy \otimes yz$$
is an injective morphism of algebras, and we henceforth identify  $U_q(\fraku \times \mathfrak{l} \times \mathfrak{u}^-)$ with its image in $U_q(\g \times \g)$. We consider the following two commuting actions of $U_q(\g \times \g) = \uqg \otimes \uqg$ on $\O_q(G \times G) = \O_q(G) \otimes \O_q(G)$, given in terms of matrix coefficients:
\begin{itemize}
\item The `internal' action: $(x_1 \otimes x_2) \triangleright (c_{f,v} \otimes c_{g,w}) = c_{f, x_1\cdot v} \otimes c_{x_2 \cdot g, w}.$
\item The `external' action: $(y_1 \otimes y_2) \triangleright (c_{f,v} \otimes c_{g,w}) = c_{y_1 \cdot f, v} \otimes c_{ g, y_2 \cdot w}.$
\end{itemize}
We consider the restriction of the internal action to the subalgebra $\uqulu$. The resulting space of invariants $\O_q(G \times G)^{\uqulu}$ carries an external action of $U_q(\g \times \g)$

\begin{theorem}\label{thm:gradedinvariants} For any $I \subseteq \Delta$, there is a $ U_q(\g \times \g)$-equivariant isomorphism of algebras $$\gr_I(\OqG)) = \O_q(G \times G)^{\uqulu}.$$ \end{theorem}

The proof of this theorem requires some set-up. 

\begin{definition} Let $\cV$ be an irreducible module for $\uqg$ with highest weight vector $v_0$. We denote by $\cV_I$ the $U_q(\mathfrak l)$-submodule of $\cV$ generated by $v_0$, that is, $\cV_I = U_q(\mathfrak l) \cdot v_0$. \end{definition}

\begin{lemma}\label{lem:vifacts} We collect the following facts:
\begin{enumerate}
\item The space $\cV_I$ is the sum of weight subspaces of $\cV$ with weights that differ from the highest weight by $\Lambda_I$.
\item The subspace $\cV_I$ of $\cV$ coincides with the $U_q(\mathfrak u_I)$-invariants: $\cV_I = \cV^{U_q(\mathfrak u_I)}$. Consequently, $\cV_I$ is $U_q(\mathfrak p_I)$-stable. 
\item The dual $(\cV_I)^*$ can be identified with the $U_q(\mathfrak{u}_I^-)$-invariants in $\cV^*$: $(\cV_I)^* = (\cV^*)^{U_q(\mathfrak{u}_I^-)}$
\item\label{lem:lrepns} As modules for $U_q(\frakl)$,  $(\cV_\lambda)_I$ and $(\cV_\mu)_I$ are isomorphic if and only if $\lambda = \mu$. 
\end{enumerate}
\end{lemma}

\begin{proof} By definition, the space $\cV_I$ is obtained from $v_0$ by successive application of the operators $F_i$ for $i \in I$. Thus, the weight vectors that appear in $\cV_I$ are precisely those whose weights differ from the highest weight by elements of $\Lambda_I$, and the first claim is established. It is straightforward to use the PBW theorem for $\uqg$ to show that $\cV_I \subseteq \cV^{U_q(\fraku_I)}$. The opposite inclusion follows from the description of $\cV_I$ given in the first statement, thus proving the second claim. For the third claim, recall that, if $v_0$ is a highest weight vector in $\cV$, then the dual vector $v_0^*$ is a lowest weight vector in $\cV^*$.  Now, $\cV_I^* =(U_q(\mathfrak{l}) \cdot v_0)^* = U_q(\mathfrak{l}) \cdot v_0^* = (\cV^*)^{U_q(\mathfrak{u}_I^-)}.$ Another way to prove the third claim is to use the second claim and match weight subspaces. For the last claim, one considers the action of the $K_i$ in $U_q(\frakl)$ and the action of the quantized  enveloping algebra $U_q(\mathfrak{l/z(l))}$ of the semisimple Lie algebra $\frakl/\frak{z(l)}$, where $\frak{z(l)}$ denotes the center of $\frakl$. \end{proof}

\begin{lemma}\label{lem:viotvi} Suppose  $x \in \cV_\lambda \otimes \cV_\mu$ is a weight vector of weight $\nu$. Then $\lambda + \mu - \nu \in \Lambda_I$ if and only if $x \in  (\cV_\lambda)_I \otimes (\cV_\mu)_I$. \end{lemma}

\begin{proof} Write $x = \sum_j v_j \otimes w_j$. Without loss of generality, we can assume that $v_j$ and $w_j$ are weight vectors of weights $\text{wt}(v_j)$  and $\text{wt}(w_j)$. Then $\text{wt}(v_j) + \text{wt}(w_j) = \nu$ for all $j$. We can write $\lambda - \text{wt}(v_j) = \sum_{i \in \Delta} n_i \alpha_i$ and $\mu - \text{wt}(w_j) = \sum_{i \in \Delta} m_i \alpha_i$ for some integers $n_i, m_i \geq 0$. Now, 
$$ \sum_{i\in \Delta} (n_i + m_i) \alpha_i = \lambda + \mu - (\text{wt}(v_j) + \text{wt}(w_j)) = \lambda + \mu -\nu.$$ The far right hand side lies in $\Lambda_I$ if and only if  $n_i = m_i = 0$ for $i \notin I$, which in turn is true if and only if $v_j \in (\cV_\lambda)_I$ and $w_j \in (\cV_\mu)_I$ for all $j$.\end{proof}

We write  $\cV_\lambda \otimes \cV_\mu = \bigoplus_{\rho \leq \lambda + \mu} \cV_\rho^{\oplus N_{\lambda \mu}^\rho}$ for the decomposition of the tensor product $\cV_\lambda \otimes \cV_\mu$ into irreducibles. 

\begin{lemma}\label{lem:subrepn} We have:
\begin{enumerate}
 \item As a $U_q(\mathfrak l)$-module, the tensor product $(\cV_\lambda)_I \otimes (\cV_\mu)_I$ is isomorphic to the direct sum of  $(\cV_\nu)_I^{\oplus N_{\lambda \mu}^\nu}$ for weights $\nu$ that differ from $\lambda + \mu$ by $\Lambda_I$.
 
 \item The $\uqg$-submodule $W$ of $\cV_\lambda \otimes \cV_\mu$ generated by $(\cV_\lambda)_I \otimes (\cV_\mu)_I$ is isomorphic to the direct sum of  $\cV_\nu^{\oplus N_{\lambda \mu}^\nu}$ for weights $\nu$ that differ from $\lambda + \mu$ by $\Lambda_I$.
\end{enumerate} \end{lemma}

\begin{proof} For each part of the lemma, we identify each of the two modules in question as a subspace of $\cV_\lambda \otimes \cV_\mu$ and prove inclusions in both directions. 

Suppose $x \in \mathcal \cV_\lambda \ot \cV_\mu$ is the highest weight vector of a $\uqg$-submodule isomorphic to $\cV_\nu$, where $\lambda + \mu -\nu \in \Lambda_I$. By Lemma \ref{lem:viotvi}, $x$ lies in $(\cV_\lambda)_I \otimes (\cV_\mu)_I$. So the $U_q(\mathfrak l)$-submodule  generated by $x$ is contained in $(\cV_\lambda)_I \otimes (\cV_\mu)_I$. We conclude that $(\cV_\lambda)_I \otimes (\cV_\mu)_I$ contains, as a $U_q(\mathfrak l)$-submodule, the direct sum of  $(\cV_\nu)_I^{\oplus N_{\lambda \mu}^\nu}$ for weights $\nu$ such that $\lambda + \mu - \nu \in \Lambda_I$. It is an immediate consequence that the $\uqg$-submodule $W$ of $\cV_\lambda \otimes \cV_\mu$ generated by $(\cV_\lambda)_I \otimes (\cV_\mu)_I$  contains the direct sum of  $\cV_\nu^{\oplus N_{\lambda \mu}^\nu}$ for weights $\nu$ that $\lambda + \mu - \nu \in \Lambda_I$. We have established one inclusion for each of the two claims. 
 
For the opposite inclusions, observe that $(\cV_\lambda)_I \otimes (\cV_\mu)_I$ is  $U_q(\mathfrak p_I)$-stable, so any element in $W$ lying outside of $(\cV_\lambda)_I \otimes (\cV_\mu)_I$ is obtained by applying the action of $F_{i}$ for $i \in \Delta \setminus I$ to elements in $(\cV_\lambda)_I \otimes (\cV_\mu)_I$. Hence, the resulting elements have lower weights than those in $(\cV_\lambda)_I \otimes (\cV_\mu)_I$. It follows that all highest weight vectors of $W$  are contained in $(\cV_\lambda)_I \otimes (\cV_\mu)_I$, and this establishes the second inclusion of the second claim. The second inclusion of the first claim now follows from Lemma \ref{lem:viotvi}, which implies that any weight vector of  $(\cV_\lambda)_I \otimes (\cV_\mu)_I$ has weight $\nu$ satisfying $\lambda + \mu - \nu \in \Lambda_I$. \end{proof}

\newcommand{\uql}{U_q(\mathfrak{l})}

\begin{proof}[Proof of Theorem \ref{thm:gradedinvariants}] The invariants in $\OqG$ for the right (resp.\ left) action of $U_q(\mathfrak{u}_I)$ (resp.\ $U_q(\fraku^-)$) can be expressed as:
$$ \O_q(G)^{U_q(\fraku_I)}  = \bigoplus_{\lambda \in \Lambda^+} \cV_\lambda^* \otimes \cV_\lambda^{U_q(\fraku_I)} = \bigoplus_{\lambda \in \Lambda^+} \cV_\lambda^* \otimes (\cV_\lambda)_I.$$
$$ {}^{U_q(\fraku_I^-)} \O_q(G)  = \bigoplus_{\lambda \in \Lambda^+} (\cV_\lambda^*)^{U_q(\fraku_I^-)} \otimes \cV_\lambda = \bigoplus_{\lambda \in \Lambda^+} (\cV_\lambda)_I^* \otimes \cV_\lambda.$$ 
The residual right action of $U_q(\frakl)$ on $\O_q(G)^{U_q(\fraku_I)}$ is on the second factor, and the residual left action of $U_q(\frakl)$ on ${}^{U_q(\fraku_I^-)} \O_q(G)$ is on the first factor. Therefore, we have the following isomorphisms of $\uqg$-bimodules:
\begin{align*} \O_q(G \times G)^{\uqulu} 
&= \left[ \O_q(G)^{U_q(\fraku_I)} \otimes \O_q(G)^{U_q(\fraku_I^-)} \right]^{U_q(\frakl)} \\ &= \left( \left[ \bigoplus_{\lambda \in \Lambda^+} \cV_\lambda^* \otimes (\cV_\lambda)_I \right] \otimes \left[ \bigoplus_{\mu \in \Lambda^+}  (\cV_\mu)_I ^* \otimes \cV_\mu \right] \right)^{U_q(\frakl)} \\ 
&= \bigoplus_{\lambda,\mu \in \Lambda^+} \cV_\lambda^* \otimes \left[ (\cV_\lambda)_I \otimes (\cV_\mu)_I ^* \right]^{\uql} \otimes \cV_\mu \\ 
&= \bigoplus_{\lambda \in \Lambda^+} \cV_\lambda^* \otimes \left[ (\cV_\lambda)_I \otimes (\cV_\lambda)_I ^* \right]^{\uql}  \otimes \cV_\lambda   
\end{align*}
The last step follows from Lemma \ref{lem:vifacts}.\ref{lem:lrepns}. Define a homomorphism $$\Phi : \gr_I(\O_q(G)) \rightarrow \O_q(G \times G)^{\uqulu} $$ $$ c_{f,v}^{\cV_\lambda} \mapsto c_{f,e_i^{(\lambda)}}^{\cV_\lambda} \otimes c_{e^i_{(\lambda)},v}^{\cV_\lambda},$$
where $v \in \cV_\lambda$, $f \in \cV_\lambda^*$, and  $\{e_i^{(\lambda)}\}$ and $\{e^i_{(\lambda)}\}$ are dual bases of $(\cV_\lambda)_I$ and $(\cV_\lambda)_I^*$. (Here we adopt Einstein notation for summing over the index $i$.) Observe that, since $\cV_\lambda$ is an irreducible module of $\uql$, the space $[ (\cV_\lambda)_I \ot  (\cV_\lambda)_I ^* ]^{\uql}$ is the span of $e_i^{(\lambda)} \otimes e^i_{(\lambda)}$. Hence, $\Phi$ is well-defined and does not depend on the choice of basis and dual basis. It is clear that $\Phi$ is an isomorphism of $\uqg$-bimodules.

We show that $\Phi$ is a homomorphism of algebras. For weights $\lambda$, $\mu$, and $\nu$, fix a basis $\gamma_1^{(\nu)}$, $\dots$, $\gamma_{N_{\lambda, \mu}^\nu}^{(\nu)}$ of $\Hom_{\uqg}(\cV_\lambda \ot \cV_\mu, \cV_\nu)^* =  \Hom_{\uqg}(\cV_\lambda^* \ot \cV_\mu^*, \cV_\nu^*)$. Choose a dual basis $\sigma_i^{(\nu)}$ of $\Hom_{\uqg}(\cV_\lambda \ot \cV_\mu, \cV_\nu)$. Specifically, we have that $\sigma_i \circ \gamma_j^*$ is the identity on $\cV_\nu$ if $i$ and $j$ are equal and zero otherwise. Choose a basis $e_k^{(\nu)}$ of $(\cV_\nu)_I$. In what follows, we will consider copies of $\cV_\nu$ indexed by $i = 1, \dots, N_{\lambda,\mu}^\nu$. We write $e_k^{(\nu,i)}$ to indicate the vector $e_k^{(\nu)}$ of the $i$th copy of $\cV_\nu$, and similarly for the dual vectors. 

Let $\lambda, \mu \in \Lambda^+$, $v \in \cV_\lambda$, $f \in \cV_\lambda^*$, $z \in \cV_\mu$, and $g \in \cV_\mu^*$. By results in Section \ref{subsec:qfiltrations} and Lemma \ref{lem:subrepn}, the algebra structure on $\gr_I(\O_q(G))$ can be described as 
$$ c_{f,v}^{\cV_\lambda} \cdot  c_{g,z}^{\cV_\mu} =  \sum_{\nu \in \lambda + \mu + \Lambda_I} \sum_{i = 1}^{N_{\lambda, \mu}^\nu} c^{\cV_\nu}_{\gamma_{i}^{(\nu)}\left(f\otimes g\right), \sigma_i^{(\nu)} \left(v\otimes z\right)}.$$
Then 
\begin{align*} 
\Phi \left( c_{f,v}^{\cV_\lambda} \cdot  c_{g,z}^{\cV_\mu} \right) &=   \sum_{ \nu \in \lambda + \mu + \Lambda_I} \sum_{i = 1}^{N_{\lambda, \mu}^\nu} \Phi\left( c^{\cV_\nu}_{\gamma_{i}^{(\nu)}\left(f\otimes g\right), \sigma_i^{(\nu)} \left(v\otimes z\right)}\right) \\ 
&= \sum_{\nu \in \lambda + \mu + \Lambda_I} \sum_{i = 1}^{N_{\lambda, \mu}^\nu}  c^{\cV_\nu}_{\gamma_{i}^{(\nu)}(f\otimes g), e_k^{(\nu, i)}} \otimes c^{\cV_\nu}_{e^k_{(\nu, i)}, \sigma_i^{(\nu, i)} (v\otimes z)}  
\end{align*}
Taking the elements $ e_k^{(\nu, i)}$ over all $i = 1, \dots, N_{\lambda, \mu}^{\nu}$ and $k = 1, \dots, \dim((\cV_\nu)_I)$, we obtain a basis for $(\cV_\nu)_I^{\oplus N_{\lambda \mu}^\nu}$. By Lemma \ref{lem:subrepn}, the direct sum of  the spaces  $(\cV_\nu)_I^{\oplus N_{\lambda \mu}^\nu}$ for all weights $\nu$ in $\lambda + \mu + \Lambda_I$,  is isomorphic (as a $U_q(\mathfrak l)$-module) to the tensor product $(\cV_\lambda)_I \otimes (\cV_\mu)_I$. Thus, a different basis for this space is given by $e_j^{(\lambda)} \ot e_\ell^{(\mu)}$ where $j = 1, \dots, \dim((\cV_\lambda)_I)$, and $\ell = 1, \dots, \dim((\cV_\mu)_I)$. Consequently, the above expression is equal to:
\begin{align*}
c^{ \cV_\lambda \ot \cV_\mu }_{ f \otimes g, e_j^{(\lambda)} \ot e_\ell^{(\mu)} } \ot c_{e^k_{(\lambda)} \ot e^\ell_{(\mu)}, v\otimes z}^{\cV_\lambda \ot \cV_\mu  } =  \left(  c^{\cV_\lambda}_{f,e_k^{(\lambda)}} \otimes c^{\cV_\lambda}_{e^k_{(\lambda)}, v} \right) \cdot \left( c^{\cV_\mu}_{g,e_\ell^{(\mu)}} \otimes c^{\cV_\mu}_{e^\ell_{(\mu)}, z} \right) = \Phi \left( c_{f,v}^{\cV_\lambda} \right) \cdot \Phi \left( c_{g,z}^{\cV_\mu} \right), 
\end{align*}
where the multiplication in the last two expressions occurs in $\O_q(G \times G)^{\uqulu}$ as a subalgebra of $\O_q(G \times G)$. \end{proof} 

\begin{example} In the extreme cases we have $\gr_{\emptyset} (\O_q(G))$ $=$ $\O_q\left(\frac{G/N \times N^-\setminus G}{T}\right)$ and $\gr_\Delta (\O_q(G))$ $=$ $\O_q(G)$. The former is the quantized coordinate algebra of the asymptotic semigroup $\text{\rm As}G$ of $G$; the asymptotic semigroup is defined in \cite{popov, VinbergAsym}. \end{example}

\begin{rmk} We make the following remarks:
\begin{enumerate}
 \item Bezrukavnikov and Kazhdan observe the classical version of the result of Theorem \ref{thm:gradedinvariants} in \cite[Remark 2.9]{BezKazhdan}. 
 \item See Section 2.5 of \cite{EvensJones} for results analogous to Lemma \ref{lem:vifacts} in the classical case. 
\end{enumerate}\end{rmk}


\section{The case of $\SL_2$}\label{sec:qsl2}

In this section, we describe  the constructions of  Section \ref{sec:quantumwc} in the case when $G= \SL_2$. 

\subsection{The algebras $\O_q(\SL_2)$ and $U_q(\mathfrak{sl}_2)$}\label{subsec:uosl2}

Fix $q \in \C^\times$. We assume that $q$ is not a root of unity. The  following discussion is adapted in part from \cite{BrownGoodearl}.

\begin{definition} The quantum $2 \times 2$ matrix algebra is the bialgebra $\O_q(\Mat_2)$ generated over $\C$ by elements $a,b,c,d$ with relations 
$$ ab = qba \qquad ac = qca \qquad bc = cb  \qquad bd = qdb  \qquad$$  $$cd = qdc \qquad  ad-da = (q - q\inv)bc,$$ and with coalgebra structure given by $$ \Delta(a) = a \otimes a + b \otimes c, \quad \Delta(b) = a \otimes b + b \otimes d, \quad \Delta(c) = c \otimes a + d \otimes c, \quad \Delta(d) =  c\otimes b + d \otimes d,$$ $$ \epsilon(a) =\epsilon(d) = 1, \quad \epsilon(b) = \epsilon(c) =0.$$ 
The quantum determinant is the (central) element $D_ q := ad-qbc$ of $\O_q(\Mat_2)$. \end{definition}

\begin{definition}\label{def:oqsl2} The quantum coordinate algebra $\O_q(\SL_2)$ of $\SL_2$ is the quotient of $\O_q(\Mat_2)$ by the ideal generated by the central element $D_q -1$, and the quantum coordinate algebra $\O_q(\GL_2)$ of $\GL_2$ is the localization of $\O_q(\Mat_2)$ at the central element $D_q$:
$$\O_q(\SL_2) = \O_q(\Mat_2)/\langle D_q -1\rangle \qquad \qquad \O_q(\GL_2) = \O_q(\Mat_2)[D_q\inv].$$
\end{definition}

We use the same notation for elements of $\O_q(\Mat_2)$ and their images in $\O_q(\SL_2)$ and $\O_q(\GL_2)$. 

\begin{lemma} The bialgebra structure on $\O_q(\Mat_2)$ descends to a bialgebra structure on $\O_q(\SL_2)$ and $\O_q(\GL_2)$. Each of the latter bialgebras is a Hopf algebra with antipode is given by:
$$ S(a) = d \qquad S(b) = - q\inv b \qquad S(c) = -q c \qquad S(d) = a.$$ \end{lemma}

\begin{definition} The quantized enveloping algebra $\Uqsl$ of $\fraksl_2$ is generated by elements $E,F,K^{\pm 1}$ subject to the relations 
$$KE = q^2 EK \qquad KF = q^{-2} FK \qquad [E,F] = \frac{K - K\inv}{q - q\inv}.$$
\end{definition}

We omit the definitions of the coalgebra structure and antipode on $\Uqsl$ from this summary (but see Proposition  \ref{prop:uqghopfstr}). For any non-negative integer $n$, there is a unique finite-dimensional irreducible $\Uqsl$-module of type 1, which we denote $\mathcal \cV_n$. 

\subsection{The Peter-Weyl filtration on $\O_q(\SL_2)$}\label{subsec:pwsl2} In the previous section, we gave a definition of $\Oqsl$ without reference to matrix coefficients; one can show that that definition coincides with the more general Defintion \ref{def:oqg}. More precisely:

\begin{theorem}\cite[Theorem I.7.16]{BrownGoodearl} The sub-Hopf algebra of $\Uqsl^{\circ}$ generated by the matrix coefficients of the representations $\mathcal \cV_n$ is isomorphic to the Hopf algebra $\Oqsl$ of Definition \ref{def:oqsl2}. Consequently, there is an isomorphism of $\Uqsl$-bimodules: $$\phi: \bigoplus_{n} \mathcal \cV_n \otimes \mathcal \cV_n^* \stackrel{\sim}{\longrightarrow} \O_q(SL_2).$$ \end{theorem}

\begin{definition}\label{def:sl2filtration} Endow $\Z$ with the dominance ordering. Define subspaces of $\O_q(\SL_2)$ by 
$$\O_q(\SL_2)_{\leq n} = \phi\left( \bigoplus_{m \leq n} \mathcal \cV_m \otimes \mathcal \cV_m^*\right).$$ \end{definition}

Thus, we have 	
$$ \O_q(\SL_2)_{\leq 0} \subseteq \O_q(\SL_2)_{\leq 2} \subseteq \O_q(\SL_2)_{\leq 4} \subseteq \cdots$$
$$ \O_q(\SL_2)_{\leq 1} \subseteq \O_q(\SL_2)_{\leq 3} \subseteq \O_q(\SL_2)_{\leq 5} \subseteq \cdots$$
and no inclusions between the two strings. The following lemma is a straightforward verification; the proof of each part is parallel to its classical analogue.

\begin{lemma}  We have the following:
\begin{enumerate}

\item  The spaces of Definition \ref{def:sl2filtration} define a filtration on $\O_q(\SL_2)$:
$$ \mu : \O_q(\SL_2)_{\leq n} \otimes \O_q(\SL_2)_{\leq m} \rightarrow \O_q(\SL_2)_{\leq n+m}.$$

\item The coproduct $\Delta$ restricts to a map $$\Delta : \O_q(\SL_2)_{\leq n} \rightarrow \O_q(\SL_2)_{\leq n} \otimes \O_q(\SL_2)_{\leq n}.$$ 

\item\label{lem:qleqn} For $n \in \Z$, consider the subspace of the free algebra $\C\langle a,b,c,d\rangle$ spanned by monomials words of length $k$ where $k \leq n$ (under the usual partial order on $\Z$) and $k \equiv n \mod 2$. The image of this subspace under the quotient map $\C\langle a,b,c,d \rangle \rightarrow \O_q(\SL_2)$ is precisely  $\O_q(\SL_2)_{\leq n}$. 
\end{enumerate}
 \end{lemma}

\begin{definition} We define the following algebras:
\begin{itemize}
 \item Let $\Sym_q^k( \C^2)$ denote the $k$th graded piece of the algebra $\C\langle x, y \rangle / \langle xy - qyx\rangle$, and set $$\P_q^1 \times \P_q^1 = \bigoplus_{k} \Sym_q^k( \C^2) \otimes \Sym_q^k( \C^2).$$
 
 \item The associated graded algebra of $\O_q(\SL_2)$ is defined as the $\Z$-graded algebra $$\text{\rm gr}(\O_q(\SL_2)) = \bigoplus_{n} \O_q(\SL_2)_{\leq n}/ \O_q(\SL_2)_{< n}.$$  
\end{itemize}
\end{definition}

Setting $q=1$ in the definition of $\P^1_q \times \P^1_q$, we obtain the homogeneous coordinate ring of $\P^1 \times \P^1$. Note that $\P^1$ is the flag variety for $\SL_2$.

\begin{prop}\label{prop:qassgrad} The associated graded $\gr(\Oqsl)$ is isomorphic to $\P^1_q \times \P^1_q$. \end{prop}

\begin{proof} Observe that there are isomorphisms
$$\text{\rm gr}(\Oqsl)) \simeq \O_q(\Mat_2)/(ad-qbc) \simeq \bigoplus_{k } \Sym_q^k( \C^2) \otimes \Sym_q^k( \C^2),$$ where the second isomorphism is given by  $$a \mapsto x \otimes u \qquad b \mapsto x \otimes w \qquad c \mapsto y \otimes u \qquad d \mapsto y \otimes w.$$ Here $(x,y)$ and $(u,w)$ denote the coordinates on the first and second copies of $\C^2$, respectively. \end{proof}

\subsection{The quantum Vinberg semigroup}\label{subsec:vinbergsl2}

\begin{definition}  The quantum Vinberg semigroup $\O_q(\VinbSL)$ for $\SL_2$ is defined  as the Rees algebra for $\Oqsl$ with the Peter-Weyl filtration. Explicitly, letting $z$ be a formal variable, $\O_q(\VinbSL)$ is the following graded subalgebra of $\O_q(\SL_2)[z]$: $$\O_q(\VinbSL)= \bigoplus_{n} \O_q(\SL_2)_{\leq n} z^n. $$ \end{definition}

\begin{prop} There is a well-defined bialgebra structure on $\VinbqSL$ given by $$\Delta : \O_q(\SL_2)_{\leq n} z^n \rightarrow  \O_q(\SL_2)_{\leq n} z^n  \otimes  \O_q(\SL_2)_{\leq n} z^n $$ $$ \Delta(fz^n) = \Delta_{\SL_2}(f) \cdot (z^n \otimes z^n),$$ and $\epsilon (fz^n) = \epsilon_{\SL_2}(f),$ where $\Delta_{\SL_2} $ and $\epsilon_{\SL_2}$ denote the coproduct and counit on $\O_q(\SL_2)$.  \end{prop}

\begin{proof} The proof is a routine computation. Note that $\VinbqSL$ is generated as an algebra by $\O_q(\SL_2)_{\leq 0} = \C\cdot 1$ and $\O_q(\SL_2)_{\leq 1} z = \{ az, bz, cz, dz\}$. We set $\Delta(az) = az \otimes az + bz \otimes cz$, $\Delta(bz) = az \otimes bz + bz \otimes dz$, etc.  \end{proof}

\begin{rmk} The algebra $\VinbqSL$ does not have an antipode. If it did, the antipode $S(z)$ of $z$ would satisfy $1 = \epsilon(z) = zS(z)$, but $z$ is not invertible in $\VinbqSL$. \end{rmk}

\begin{prop}\label{prop:vinb=matrix} There is an isomorphism of bialgebras $\O_q(\VinbSL) \simeq \O_q(\Mat_2)$.\end{prop}

\begin{proof} The proof is analogous to the proof in the classical case. By Lemma \ref{lem:qleqn}, an element of $\O_q(SL_2)_{\leq n}$ can be represented by a word in $a,b,c,d$ of length $k$ where $k \leq n$ (in the usual partial order on $\Z$) and $k \equiv n \mod 2$. Similarly, an element of $\O_q(SL_2)_{\leq n}z^n$ can be represented by such a word together with a factor of $z^n$. Using the commutation relations between the generators of $\Oqsl$ and the relation relation $z^2 = (az)(dz) -q(bz)(cz),$ one sees that such a word lies in the span  of the words $(az)^{k_1^\prime} (bz)^{k_2^\prime} (cz)^{k_3^\prime} (dz)^{k_4^\prime} $ with  $k_1^\prime+k_2^\prime + k_3^\prime + k_4^\prime  \leq n$ (in the usual partial order on $\Z$). Hence, $\O_q(\VinbSL)$ is generated by the elements $az$, $bz$, $cz$, $dz$. The relations and coproduct on $\O_q(\VinbSL)$ coincide with those of $\O_q(\Mat_2)$. \end{proof}

Observe that $\VinbqSL$ contains $z^2$, but does not contain $z$. The following corollary can be compared to the decomposition in Equation \ref{eq:SL2orbits} of Section \ref{subsec:sl2}. 
 
\begin{cor} There are isomorphisms of algebras  $$(\O_q(\VinbSL)[(z^2)\inv])^{\C^\times} \simeq \Oqsl \qquad \text{and} \qquad  \O_q(\VinbSL)/ (z^2) \simeq \P^1_q \times \P^1_q.$$
\end{cor}

\begin{proof} Under the isomorphism of Proposition \ref{prop:vinb=matrix}, the element $z^2 \in \O_q(\VinbSL)$ corresponds to the element $D_q \in \O_q(\Mat_2)$. Thus,
$$(\O_q(\VinbSL)[(z^2)\inv])^{\C^\times} = (\O_q(\Mat_2)[D_q\inv])^{\C^\times} = \O_q(\GL_2)^{\C^\times} = \Oqsl.$$ It is straightforward to show that $\VinbqSL/(z^2) = \gr(\Oqsl),$ and hence the second isomorphism is a consequence of Proposition \ref{prop:qassgrad}. \end{proof}

\bibliographystyle{plain}

\end{document}